\documentclass{cmslatex}
\usepackage[paperwidth=7in, paperheight=10in, margin=.875in]{geometry}
 \usepackage[backref,colorlinks,linkcolor=red,anchorcolor=green,citecolor=blue]{hyperref}
\usepackage{amsfonts,amssymb}
\usepackage{amsmath}
\usepackage{graphicx}
\usepackage{cite}
\usepackage{enumerate}
\usepackage{caption}
\usepackage{subfigure}

\renewcommand{\theequation}{\arabic{section}.\arabic{equation}}

   \allowdisplaybreaks
\begin{document}
 \title{The interaction of elementary waves for nonisentropic flow in a variable cross-section duct\thanks{Received 5 November 2019, accepted 21 December 2020.}}


          \author{Qinglong Zhang\thanks{School of Mathematics and Statistics, Ningbo University, Ningbo, 315211, P.R.China. Email: zhangqinglong@nbu.edu.cn. This work is sponsored by K.C. Wong Magna Fund in Ningbo University. }}
          \
          
         \pagestyle{myheadings} \markboth{INTERACTIONS FOR NONISENTROPIC FLOW}{QINGLONG ZHANG} \maketitle

          \begin{abstract}
              The interaction of elementary waves for isentropic flow in a variable cross-section duct is obtained (\cite{ShengZhang}). The authors have discussed rarefaction wave or shock wave interacts with stationary wave. In this paper, we extend their results to the nonisentropic flow. It can be proved that if one changes $(u,\rho)$ plane in isentropic flow with $(u,p)$ plane in nonisentropic case, the interaction results of rarefaction wave or shock wave with stationary wave can be moved parallel from the isentropic case when contact discontinuity is involved. Thus we mainly focus on the interactions between contact discontinuity and stationary wave.  Some numerical results are given to verify our analysis. The results can apply to the interaction of more complicated wave patterns.
          \end{abstract}


          \section{Introduction}\label{sec1}
          
The equations of nonisentropic flow in a variable cross-section duct are given by
\begin{equation}\label{1.1}
\left\{
\begin{array}{l}
(a\rho)_t+(a\rho u)_{x}=0,\\
(a\rho u)_t+(a\rho u^2+ap)_x=pa_x,\\
(a\rho E)_t+(au(\rho E+p))_x=0,\\
a_t=0, \\
\end{array}
\right.
\end{equation}
where $\rho, u, p$ and $E$ represent the density, velocity, pressure and total energy of the fluid, respectively. $E=e+\frac{1}{2}u^2$ with the internal energy $e=\frac{p\tau}{\gamma-1}$.The state equation is given by $p=\kappa(S) \rho^{\gamma}$, where $\kappa(S)$ is a variant corresponding to $S$ and $1<\gamma<3$. Generally, $a(x)$~is given as a prior, here we view it as a variant which is independent of time (\cite{Andrianov1,Goatin5,LeFloch11}).

System \eqref{1.1} is not conservative because of the existence of source term, which can be seen as nonconservative product (\cite{LeflochTzavaras13}). A general definition on the nonconservative product can be found in  \cite{DalLeflochMurat4}. The discretization of source term plays an important role in the numerical approximations in many areas, for example, the nozzle flow model \cite{Andrianov1}, the shallow water equations \cite{Jin} and the multiphase flow models (\cite{Andrianov2,Baer2,KronerThanh8,KronerLeFlochThanh9}), to name just a few. 

In 2003, LeFloch and Thanh (\cite{LeflochThanh12}) solved the isentropic flow in a variable cross-section duct by dividing $(u,\rho)$ plane with the coinciding characteristic curves. In each area, system \eqref{1.1} can be viewed as strictly hyperbolic. The Riemann problem of system (\ref{1.1}) was studied by Andrianov and Warnecke (\cite{Andrianov1})  and  Thanh (\cite{Thanh17}), where the admissible criterion is proposed to select a physical relevant solution.

Recently, Sheng and Zhang \cite{ShengZhang} investigated the interaction of elementary waves of isentropic flow in variable cross-section duct. They give the results when the rarefaction wave or shock wave interacts with the stationary wave. The interaction results apart from the stationary wave can be found in \cite{ChangHsiao3,Smoller}. In this paper, we aim to extend their results to noninsentropic case. While the interaction results of isentropic flow can be moved parallel to nonisentropic case when contact discontinuity is involved, we devote to the interaction of contact discontinuity with stationary wave. The characteristic analysis method is used to analysis all the possible cases. Besides, some numerical results are given to support our analysis. We believe that the interaction results can apply to the interactions of more complicated wave patterns.

This paper is organized as follows. In section \ref{sec2}, we recall the characteristic analysis method and give the elementary waves. In section \ref{sec3}, we mainly discuss the interaction results of contact discontinuity with the stationary wave.  The initial states in both supersonic area and subsonic area are considered. Some numerical results are given in section \ref{sec4} to verify our analysis.

\section{Preliminaries}\label{sec2}
\subsection{Characteristic analysis and elementary waves}
If we take $(p,S)$ as independent variables, $\rho$ can be viewed as the function of $(p,S)$
\begin{equation}\label{2.1}
\rho=\rho(p,S)=\left(\frac{p}{\gamma-1}{\rm exp}(\frac{S_*-S}{C_v})\right)^{1/\gamma},
\end{equation}
where $S_*, C_v$ are constants.
Denote $U=(p,u,S,a)$, when considering a smooth solution, system \eqref{1.1} can be rewritten as
\begin{equation}\label{2.2}
\partial_t U+A(U)\partial_x U=0,
\end{equation}
where
$$
A=\left(
\begin{array}{cccc}
\rho c^2 &u & 0 & \frac{\rho uc^2}{a} \\
 u &1/\rho & 0 & 0 \\
0 &0 & u & 0\\
0 & 0 & 0 & 0
\end{array}
 \right).
$$
The matrix $A$ has four eigenvalues
\begin{equation}\label{2.3}
\lambda_1=u-c, \quad \lambda_2=u, \quad \lambda_3=u+c,\quad \lambda_4=0,
\end{equation}
where $c=\sqrt{p^{\prime}(\rho)}$. The corresponding right eigenvectors are
\begin{equation}\label{2.4}
\overrightarrow{r}_1=(\rho,-1/c,0,0)^T,~\overrightarrow{r}_2=(0, 0,1,0)^T,~
\overrightarrow{r}_3=(\rho,1/c,0,0)^T,~ \overrightarrow{r}_4=(-2\rho c^2,2c^2,0,1)^T.
\end{equation}
The 2- and 4-characteristic fields are linearly degenerate, while the 1- and 3-characteristic fields are genuinely nonlinear:
\begin{equation}\label{2.5}
-\triangledown \lambda_1(U)\cdot r_1(U)=\triangledown \lambda_3(U)\cdot r_3(U)=\frac{1}{2\sqrt{p^{'}(\rho)}}(\rho p''(\rho)+2p'(\rho))>0.
\end{equation}
System \eqref{1.1} is not strictly hyperbolic because $\lambda_1,\lambda_2$ and $\lambda_3$ may coincide with $\lambda_4$. More precisely, setting
\begin{equation}\label{2.6}
\Gamma_{\pm}:u=\pm\sqrt{\gamma \kappa(S)^{1/\gamma}}~p^{\frac{\gamma-1}{2\gamma}}, \quad \Gamma_0: u=0,
\end{equation}
one can see that
\begin{equation}\label{2.7}
\lambda_1=\lambda_4 \quad {\rm on} \quad \Gamma_{+},\qquad
\lambda_3=\lambda_4 \quad {\rm on} \quad \Gamma_{-},\qquad
\lambda_2=\lambda_4 \quad {\rm on} \quad \Gamma_{0}.
\end{equation}
In the three-dimensional space of the coordinates $(p,u,S)$ where $a$ is constant, the three surfaces $\Gamma_{\pm}$ and $\Gamma_0$ separate the space into four regions. For convenience, we will view them as $D_{1}, D_2, D_3$ and $D_4$:
\begin{equation}\label{2.8}
\begin{array}{llll}
D_1  =\Big\{(p,u,S,a)\big|u>\sqrt{\gamma \kappa(S)^{1/\gamma}}~p^{\frac{\gamma-1}{2\gamma}}~\Big\},  \\
D_2  =\Big\{(p,u,S,a)\big|0<u<\sqrt{\gamma \kappa(S)^{1/\gamma}}~p^{\frac{\gamma-1}{2\gamma}}~\Big\},\\
D_3  =\Big\{(p,u,S,a)\big|-\sqrt{\gamma \kappa(S)^{1/\gamma}}~p^{\frac{\gamma-1}{2\gamma}}<u<0~\Big\},\\
D_4  =\Big\{(p,u,S,a)\big|u<-\sqrt{\gamma \kappa(S)^{1/\gamma}}~p^{\frac{\gamma-1}{2\gamma}}~\Big\}.
\end{array}
\end{equation}
In each of the region, the system is strictly hyperbolic and one has
\begin{equation}\label{2.9}
\begin{array}{llll}
&\lambda_1>\lambda_4, &{\rm in} & D_1,\\
&\lambda_1<\lambda_4<\lambda_2, &{\rm in} & D_2,\\
&\lambda_2<\lambda_4<\lambda_3, &{\rm in} & D_3,\\
&\lambda_3<\lambda_4, &{\rm in} & D_4.
\end{array}
\end{equation}

\subsection{The rarefaction waves}
First, we look for self-similar solutions. The Riemann invariants of each characteristic can be computed by
\begin{equation}\label{2.10}
\left\{
\begin{array}{lll}
\lambda_1=u-c: &\big\{a,u+\frac{2c}{\gamma-1},S\big\},  \\
\lambda_2=u: &\big\{a, u,p\big\},   \\
\lambda_3=u+c: &\big\{a,u-\frac{2c}{\gamma-1},S\big\},\\
\lambda_4=0: &\big\{a\rho u, \frac{u^2}{2}+\frac{c^2}{\gamma-1},S\big\}.
\end{array}
\right.
\end{equation}
The cross-section $a(x)$ remains constant across rarefaction wave, system \eqref{1.1} degenerates to the gas dynamic equations
\begin{equation}\label{2.11}
\left\{
\begin{array}{l}
\rho_t+(\rho u)_{x}=0,\\
(\rho u)_t+(\rho u^2+p)_x=0,\\
(\rho E)_t+(u(\rho E+p))_x=0.
\end{array}
\right.
\end{equation}
For a given left hand state $(p_0,u_0,S_0,a_0)$, we determine the 1-wave and 3-wave rarefaction curves that can be connected on the right by
\begin{equation}\label{2.12}
\left\{
\begin{array}{l}
\displaystyle R_1(U,U_0):~ u=u_0-\int_{p_0}^{p} \frac{\sqrt{p'(\rho)}}{\rho}\,{\rm d}p=u_0-\frac{2\sqrt{\gamma\kappa^{1/\gamma}}}{\gamma-1}
\left(p^{\frac{\gamma-1}{2\gamma}}-p_0^{\frac{\gamma-1}{2\gamma}}\right),\quad p<p_0,\\[8pt]
\displaystyle R_3(U,U_0):~ u=u_0+\int_{p_0}^{p} \frac{\sqrt{p'(\rho)}}{\rho}\,{\rm d}p=u_0+\frac{2\sqrt{\gamma\kappa^{1/\gamma}}}{\gamma-1}
\left(p^{\frac{\gamma-1}{2\gamma}}-p_0^{\frac{\gamma-1}{2\gamma}}\right),  \quad p>p_0.\\
\end{array}
\right.
\end{equation}
\subsection{The stationary waves}
The Rankine-Hugoniot relation associated with the last equation of \eqref{1.1} is that
\[-\sigma[a]=0,\]
where $[a]:~=a_1-a_0$ is the jump of the cross-section $a$. One can derive the conclusions:\\
1)~~$\sigma=0:$~the shock speed vanishes, here we assume $[a]\not=0$ and called stationary contact discontinuity; \\
2)~~$\sigma\neq 0:$~the cross-section~$a$~remains constant across the non-zero speed shocks.\\
Across the stationary contact discontinuity, the Riemann invariants remain constant, from the last equation of \eqref{2.5}, the right hand states $(p,u,S,a)$ connected with the left hand state $(p_0,u_0,S_0,a_0)$ should satisfy
\begin{equation}\label{2.13}
\left\{
\begin{array}{l}
a_0 \rho_0 u_0 =a\rho u,\\
\displaystyle \frac{u_0^{2}}{2}+\frac{\kappa \gamma}{\gamma-1} \rho_0^{\gamma-1}= \frac{u^{2}}{2}+\frac{\kappa \gamma}{\gamma-1} \rho^{\gamma-1},\\[8pt]
\displaystyle \frac{p_0}{\rho_0^{\gamma}}=\frac{p}{\rho^{\gamma}}=\kappa(S).
\end{array}
\right.
\end{equation}
By solving \eqref{2.13}, we have the following results.
\begin{lem}\label{lem2.1}
Given the left hand state $U_0=(p_0,u_0,S_0,a_0)$, \eqref{2.13} has at most two solutions $U_*=( p_*,u_*,S_0, a)$ and $U^{*}=(p^{*},u{*},S_0, a)$ for any $a>0$, if and only if $a\geq \displaystyle a_{\rm min}(U_0)$, where
\begin{equation}\label{2.14}
a_{\rm min}(U_0)=\frac{a_0 \rho_0 |u_0|}{\sqrt{\kappa \gamma}\rho_m^{\frac{\gamma+1}{2}}}
~~ {\rm and} ~~ \rho_m=\left(\frac{\gamma-1}{\kappa\gamma(\gamma+1)}u_0^{2}+\frac{2}{\gamma+1} \rho_0^{\gamma-1}\right)^{\frac{1}{\gamma-1}}.
\end{equation}
More precisely,\\
1) If $a< a_{\rm min}(U_0)$, \eqref{2.13} has no solution, so there is no stationary waves.\\
2) If $a> a_{\rm min}(U_0)$, there are two points $U_*, U^{*}$ satisfying \eqref{2.13}, which can connect with $U_0$ by stationary waves.\\
3) If $a=a_{\rm min}(U_0)$, $U_*$ and $U^{*}$ coincide.
\end{lem}
\noindent%
The proof is straightforward. For the details, we refer to \cite{LeflochThanh12,Thanh17} and don't repeat here.

Across the stationary contact discontinuity denoted by $S_0(U;U_0)$, the states $U_*=(p_*,u_*,S_0,a)$ and $U^*=(p^{*},u^{*},S_0,a)$ have the following properties
\begin{align}
S_0(U;U_0)=\left\{\begin{array}{lll}
S_0(U_*;U_0), &|u_*|>c_*,\\
S_0(U^*;U_0), &|u^*|<c^*,
\end{array}\right.
~~{\rm more~ precisely,}\\
S_0(U;U_0):~\left\{\begin{array}{lll}
 (u_*,p_*,S_0,a)\in\left\{
 \begin{array}{lll}
 D_1,&u_0>0,\\
 D_4,&u_0<0,
 \end{array}\right.\\
 (u^{*},p^{*},S_0,a)\in  \left\{
 \begin{array}{lll}
 D_2,&u_0>0,\\
 D_3,&u_0<0.
 \end{array}\right.
\end{array}\right.
\end{align}

As shown in \cite{LeflochThanh12}, the Riemann problem for \eqref{1.1} may admit up to a one-parameter family of solutions. This phenomenon can be avoided by requiring Riemann
solutions to satisfy an admissibility criterion: monotone condition on the component $a$. Followed by \cite{Andrianov1,LeflochThanh12} and \cite{Thanh17}, we
impose the following global entropy condition on stationary wave of \eqref{1.1}.\\
{\bf Global entropy condition.} Along the stationary curve $S_0(U;U_0)$ in the $(u,p)$-plane, the cross-section area $a$ obtained from \eqref{2.13} is a monotone function of $\rho$.

Under the global entropy condition, we call the stationary contact discontinuity as stationary wave and have the following results.
\begin{lem}\label{lem2.2}
Global entropy condition is equivalent to the statement that any
stationary wave has to remain in the closure of only one domain $D_i, i=1,2,3,4$.
\end{lem}

\subsection{The shock waves and contact discontinuities}
For the non-zero speed shocks, the left hand state $U_0=(p_0,u_0,S_0,a_0)$ and the right hand state $U=(p,u,S,a)$ are connected by the Rankine-Hugoniot relations corresponding to \eqref{2.11}
\begin{equation}\label{2.15}
\left\{\begin{array}{ll}
-\sigma[\rho]+[\rho u]=0,\\
-\sigma[\rho u]+[\rho u^2+p(\rho)]=0,\\
-\sigma[\rho E]+[\rho uE+up]=0,
\end{array}\right.
\end{equation}
which is equivalent to
\begin{equation}\label{2.16}
\sigma_i(U, U_0)=u_0\mp\left(\rho\rho_0\frac{[p]}{[\rho]}\right)^{1/2},\quad i=1,3.
\end{equation}
When $[p]=0$, $[\rho]\neq0$, we have the contact discontinuity corresponding to $\sigma_2(U, U_0)$, which is given by $J(U,U_0):~u=u_0,p=p_0,\rho \neq \rho_0$.

A shock wave should satisfy the Lax shock conditions (\cite{Lax10})
\begin{equation}\label{2.17}
\lambda_i(U)<\sigma_i(U,U_0)<\lambda_i(U_0),\quad i=1,3.
\end{equation}
Using the Lax shock conditions, the 1-and 3-families of shock waves with non-zero speed connecting a given left hand state $U_0=(p_0,u_0,S_0,a_0)$ to the right hand  state $U=(p,u,S,a)$ are
\begin{equation}\label{2.18}
S_1(U,U_0):\left\{ \begin{array}{lll}\displaystyle \tau=\frac{\tau_0(\mu^2p+p_0)}{p+\mu^2p_0}, ~~~~~{\rm where}\quad \mu^2=\frac{\gamma-1}{\gamma+1},\quad \tau=\frac{1}{\rho},\\[8pt]
\displaystyle u=u_0-(p-p_0)\sqrt{\frac{(1-\mu^2)\tau_0}{p+\mu^2 p_0}},~~~~p>p_0,\\
\end{array}
\right.\\
\end{equation}
\begin{equation}\label{2.19}
S_3(U,U_0):\left\{\begin{array}{lll} \displaystyle \tau=\frac{\tau_0(\mu^2p+p_0)}{p+\mu^2p_0},\\[8pt]
\displaystyle u=u_0+(p-p_0)\sqrt{\frac{(1-\mu^2)\tau_0}{p+\mu^2 p_0}},~~~~p<p_0.\\
\end{array}
\right.
\end{equation}
The 1- and 3-shock wave speeds ~$\sigma_i(U,U_0) (i=1, 3)$ may change their signs along the shock curves in the $(u,p)$ plane, more precisely,
\begin{equation}\label{2.20}
\sigma_1(U,U_0)~\left\{\begin{array}{lll}
~~<0, &U_0\in D_2\cup D_3\cup D_4,\\
\left.\begin{array}{ll}
<0, &p>\tilde{p}_0,\\
=0, &p=\tilde{p}_0,\\
>0, &p_0<p<\tilde{p}_0,
\end{array}\right\} &U_0\in D_1,
\end{array}\right.
\end{equation}
and
\begin{equation}\label{2.21}
\sigma_3(U,U_0)~\left\{\begin{array}{lll}
~~>0, &U_0\in D_1\cup D_2\cup D_3,\\
\left.\begin{array}{ll}
>0, &p>\bar{p}_0,\\
=0, &p=\bar{p}_0,\\
<0, &p_0<p<\bar{p}_0,
\end{array}\right\} &U_0\in D_4,
\end{array}\right.
\end{equation}
where $\widetilde{U}_0=(\tilde{p}_0,\tilde{u}_0,\tilde{S}_0,a)\in D_2, \overline{U}_0=( \bar{p}_0, \bar{u}_0,\bar{S}_0,a)\in D_3$.

Let us define the backward and forward wave curves
\begin{align*}
W_1(p;U_0) & =\left\{\begin{array}{ll}
R_1(p;U_0), &p<p_0,\\
S_1(p;U_0), &p>p_0,
\end{array}\right.\\
W_3(p;U_0) & =\left\{\begin{array}{ll}
R_3(p;U_0), &p>p_0,\\
S_3(p;U_0), &p<p_0,
\end{array}\right.
\end{align*}
The wave curve $W_1(p;U_0)$ is strictly decreasing and convex in the $(u,p)$ plane, while the wave curve $W_3(p;U_0)$ is strictly increasing and concave.

The elementary waves of \eqref{1.1} consists of rarefaction waves ($W_1(p;U_0)$), shock waves ($W_3(p;U_0)$), contact discontinuities ($J(U,U_0)$) and stationary waves ($S_0(U,U_0)$).

Now we turn to the interaction of elementary waves for \eqref{1.1}. Since in \cite{ShengZhang}, the authors have already discussed the isentropic case. For nonisentropic case, the result can be moved parallel if one changes the $(u,\rho)$ plane into $(u,p)$ plane as contact discontinuity is involved. Thus we mainly focus on the interaction of contact discontinuity with the stationary wave. The results are given in the following section.

\section{The interactions of contact discontinuity with stationary wave}\label{sec3}

To study the contact discontinuity interacts with the stationary wave, we consider the initial value problem \eqref{1.1} with
\begin{equation}\label{3.1}
(p,u,S,a)\Big|_{t=0}=
\left\{\begin{array}{lll}
U_-=(p_-,u_-,S_-,a_0), &x<x_1,\\
U_m=(p_m,u_m,S_m,a_0), &x_1<x<x_2,\\
U_+=(p_+,u_+,S_+,a_1), &x>x_2.
\end{array}\right.
\end{equation}
Here $U_-$ and $U_m$ are connected by a contact discontinuity $J$,  $U_m$ and $U_+$ are connected by a stationary wave $S_0$. That is 
\begin{equation}\label{3.2}
\left\{
\begin{array}{lll}
U_m\in J(U,U_-):\quad u_-=u_m>0, \quad p_-=p_m,\quad \rho_-\neq \rho_m, \\[5pt]
U_+ \in S_0(U,U_m):~~\left\{
\begin{array}{lll}
a_0 \rho_m u_m=a_1 \rho_+ u_+,\\[5pt]
\displaystyle \frac{u_m^2}{2}+\frac{c_m^2}{\gamma-1}=\frac{u_+^2}{2}+\frac{c_+^2}{\gamma-1},\\[12pt]
\displaystyle \frac{p_m}{\rho_m^{\gamma}}=\frac{p_+}{\rho_+^{\gamma}}=\kappa_+. \\
\end{array}
\right.
\end{array}
\right.
\end{equation}

In the following part, we use the characteristic analysis method to discuss the interaction results. To begin with, we are interested in the properties of the stationary wave curve $S_0(U_1,U_0)$ which is formulated by making the left-hand state $U_0\in J(U_m,U)$, the result is shown in the following lemma.
\begin{lem}\label{lem3.1}
Denote $U_0\in J(U_m,U)$ as the left hand state, the right hand state $U_1$ that can be connected with $U_0$ by the stationary wave  $S_0(U_1,U_0)$ is on the curve $l(u_+,p_+): up^{\frac1{\gamma}}=u_+p_+^{\frac1{\gamma}}$.
\end{lem}
\begin{proof} On one hand, from the assumption, one has
\begin{equation}\label{3.3}
\left\{
\begin{array}{lll}
U_0\in J(U_m,U): \quad u_0=u_m>0, \quad p_0=p_m,\quad  \rho_0\neq \rho_m,\\[5pt]
U_1 \in S_0(U,U_0):\quad \left\{
\begin{array}{lll}
a_0 \rho_0 u_0=a_1 \rho_1 u_1,\\[5pt]
\displaystyle \frac{u_0^2}{2}+\frac{c_0^2}{\gamma-1}=\frac{u_1^2}{2}+\frac{c_1^2}{\gamma-1},\\[8pt]
\displaystyle \frac{p_0}{\rho_0^{\gamma}}=\frac{p_1}{\rho_1^{\gamma}}=\kappa. \\
\end{array}
\right.
\end{array}
\right.
\end{equation}
On the other hand, we have \eqref{3.2} holds. Combining \eqref{3.2} and \eqref{3.3} to yield
\begin{equation}\label{3.4}
\left\{
\begin{array}{ll}
\displaystyle \frac{\rho_m}{\rho_0}=\frac{\rho_+ u_+}{\rho_1 u_1},\\[12pt]
\displaystyle \left(\frac{\rho_0}{\rho_m}\right)^{\gamma}=\left(\frac{\rho_1}{\rho_+}\right)^{\gamma} \frac{p_+}{p_1},
\end{array}
\right.
\end{equation}
it follows the conclusion
\begin{equation}\label{3.5}
u_1p_1^{\frac1{\gamma}}= u_+ p_+^{\frac1{\gamma}}.
\end{equation}
\end{proof}

The key observation here is that when the left hand state $U_0\in J(U_m,U)$, the right hand state $U_1\in S_0(U,U_0)$ forms a curve which can be parameterized as a function of $\rho_0$. This becomes our starting point to investigate the interaction results. Since the states on the two sides of the stationary wave remain in one domain from lemma \ref{lem2.1}, it is natural to classify the interaction results according to the intermediate state $U_m$ is supersonic or subsonic, which is represented by

 
\begin{equation}\label{3.5'}
u_m>c_m~{\rm and}~\left\{\begin{array}{l}
u_->c_-,\\
u_-<c_-,
\end{array}\right.\quad{\rm or}\quad
u_m<c_m~{\rm and}~\left\{\begin{array}{l}
u_->c_-,\\
u_-<c_-.
\end{array}\right.
\end{equation}
We will discuss them case by case.

{\bf Construction 1.  $u_m>c_m$ and $u_->c_-$}.  \space In this case,  as the contact discontinuity touches the stationary wave, $U_-$ will jump to $U_{-*}$ first as both $W_1$ and $W_3$ have positive speeds. To further determine the interaction results, it is essential to judge the relative positions of $p_{-*}$ and $p_{+}$, which is given in the following lemma.

\begin{lem}\label{lem3.2}
When $U_m$ and $U_-$ are both supersonic, i.e., $u_m>c_m, u_->c_-$, we have
 \begin{equation}
 p_{-*}<p_+ ~~ {\rm if} ~~ \rho_-<\rho_m \quad {\rm and }
 \quad p_{-*}>p_+ ~~{\rm if} ~~~ \rho_->\rho_m. 
 \end{equation}
 
\end{lem}

\begin{proof}
Assume that $U_0$ and $U_1$ are defined in lemma \ref{lem3.1}. First, differential \eqref{3.3} on both sides, one has

\begin{equation}\label{3.6}
\left\{
\begin{array}{l}
a_0 u_m {\rm d}\rho_0=a_1 \rho_1 {\rm d}u_1+a_1 u_1 {\rm d}\rho_1,\\[4pt]
\displaystyle \frac{\gamma}{\gamma-1}p_m(-\frac{1}{\rho_0^2}){\rm d}\rho_0=u_1{\rm d}u_1+\frac{\gamma}{\gamma-1}\frac{1}{\rho_1}{\rm d}p_1-\frac{\gamma}{\gamma-1}\frac{p_1}{\rho_1^2}{\rm d}\rho_1,\\[7pt]
\displaystyle  \gamma p_1\rho_0^{\gamma-1} {\rm d}\rho_0=\gamma p_m \rho_1^{\gamma-1}{\rm d}\rho_1-\rho_0^{\gamma}{\rm d}p_1. \\
\end{array}
\right.
\end{equation}
Insert the third equation of \eqref{3.3} to the third equation of \eqref{3.6}, one obtains
\begin{equation}\label{3.7}
{\rm d}\rho_1=\frac{\rho_1}{\gamma p_1}{\rm d}p_1+\frac{\rho_1}{\rho_0}{\rm d}\rho_0.
\end{equation}
Then by substituting \eqref{3.7} into the first two equations of  \eqref{3.6}, we have after arranging terms that
\begin{equation}\label{3.8}
\left\{
\begin{array}{l}
\displaystyle {\rm d}u_1=-\frac{u_1}{\gamma p_1}{\rm d}p_1,\\[9pt]
\displaystyle \frac{\gamma}{\gamma-1}\left(\frac{p_1}{\rho_0\rho_1}-\frac{p_m}{\rho_0^2}\right){\rm d}\rho_0=u_1{\rm d}u_1+\frac{1}{\rho_1}{\rm d}p_1. \\
\end{array}
\right.
\end{equation}
By eliminating ${\rm d}u_1$ in \eqref{3.8}, one gets
\begin{equation}\label{3.9}
\frac{{\rm d}p_1}{{\rm d}\rho_0}=\frac{\gamma^2 p_m p_1(\rho_0^{\gamma-1}-\rho_1^{\gamma-1})}{(\gamma-1)\rho_0^{\gamma+1}(u_1^2-c_1^2)}.
\end{equation}
From the assumption, one one hand, we have $u_0>c_0$ since $U_0\in J(U_m,U_-)$, which indicates that $p_1<p_0=p_m$ from the property of stationary wave. On the other hand, \eqref{3.3} tells $\displaystyle \frac{p_0}{\rho_0^{\gamma}}=\frac{p_m}{\rho_0^{\gamma}}=\frac{p_1}{\rho_1^{\gamma}}$, it follows that $\rho_1<\rho_0$. Thus from \eqref{3.9}, we have $\displaystyle \frac{{\rm d}p_1}{{\rm d}\rho_0}>0$ in this case. Besides, as $U_0=U_m$, $U_1=U_+$, we conclude as follows.\\
1) If $\rho_0=\rho_-<\rho_m$, then $p_1=p_{-*}<p_+$, \\
2) If $\rho_0=\rho_->\rho_m$, then $p_1=p_{-*}>p_+$.\\
Thus we prove the lemma.
\end{proof}

Based on lemma \ref{3.2}, we now give the interaction results as follows.

\begin{lem}\label{lem3.3} When $U_m$ and $U_-$ are both supersonic, $U_-$ jumps to $U_{-*}$ as the contact discontinuity touches the stationary wave. More specifically:

\noindent
Case 1. $\rho_->\rho_m$, the interaction results have two subcases (see Fig. 3.1.):

Subcase 1. If $\displaystyle u_{-*}+\frac{2}{\gamma-1}c_{-*}>u_{+}-\frac{2}{\gamma-1}c_{+}$, then the result is 
\begin{equation}
J(U_m,U_-)\oplus S_0(U_+,U_m)\rightarrow S_0(U_{-*},U_-) \oplus R_1(U_2,U_{-*})  \oplus J(\overline{U}_2,U_2) \oplus R_3(U_+,\overline{U}_2).
\end{equation}

Subcase 2.  If $\displaystyle u_{-*}+\frac{2}{\gamma-1}c_{-*}\leq u_{+}-\frac{2}{\gamma-1}c_{+}$, there exists a vacuum. The result is 
\begin{equation}
 J(U_m,U_-)\oplus S_0(U_+,U_m)\rightarrow S_0(U_{-*},U_-) \oplus R_1({\rm Vacuum},U_{-*})  \oplus  R_3(U_+,{\rm Vacuum}).
\end{equation}

\noindent
Case 2. $\rho_-<\rho_m$, the interaction result is (see Fig. 3.2.):
\begin{equation}
J(U_m,U_-)\oplus S_0(U_+,U_m)\rightarrow S_0(U_{-*},U_-) \oplus S_1(U_3,U_{-*})  \oplus J(\overline{U}_3,U_3) \oplus S_3(U_+,\overline{U}_3).
\end{equation}
Here $``\oplus"$ means ``follows".
\end{lem} 

\begin{figure}[htbp]
\subfigure{
\begin{minipage}[t]{0.4\textwidth}
\centering
\includegraphics[width=0.95\textwidth]{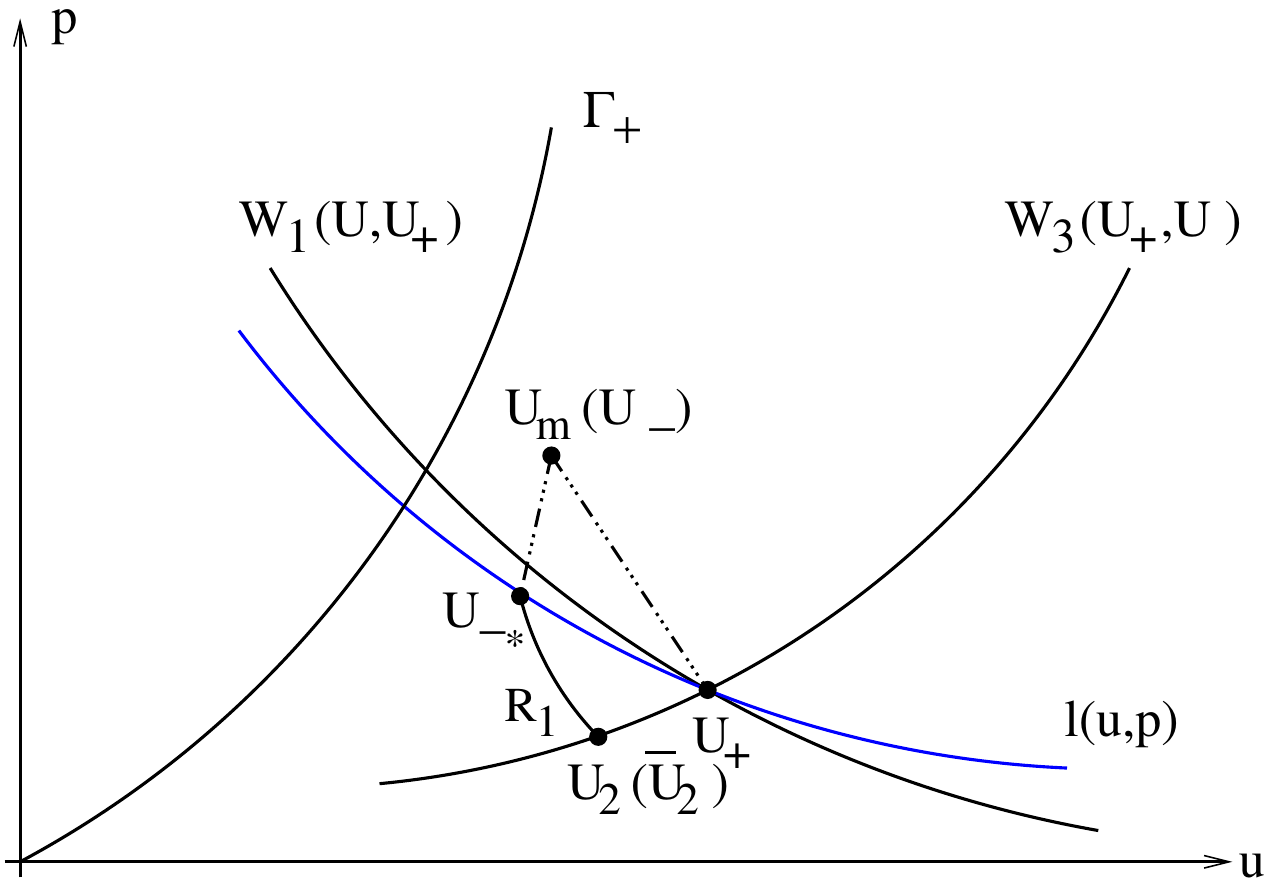}
\end{minipage}
}
\subfigure{
\begin{minipage}[t]{0.4\textwidth}
\centering
\includegraphics[width=0.95\textwidth]{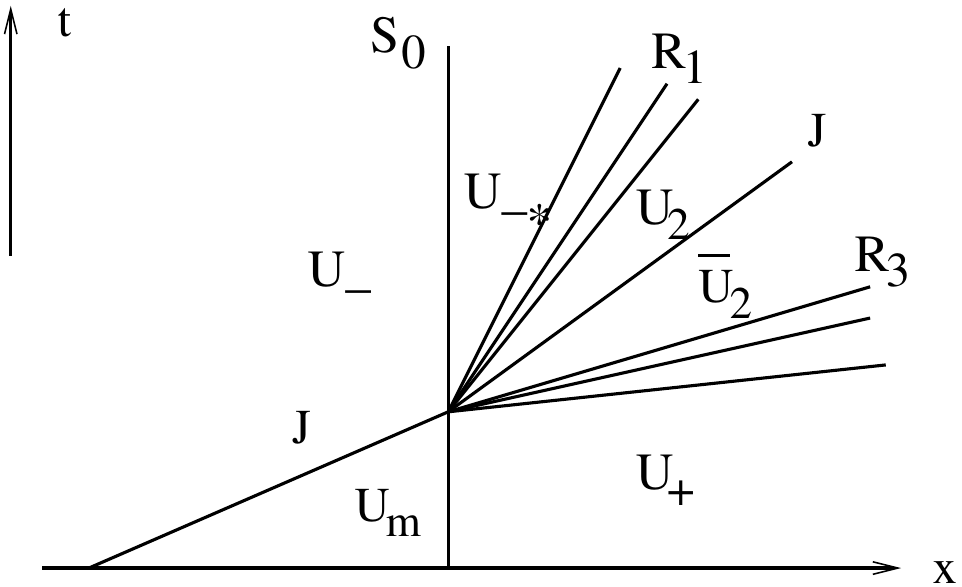}
\end{minipage}
}
\caption*{Fig. 3.1. Case 1. $u_m>c_m, u_->c_-$ and $\rho_->\rho_m$.}
\end{figure}

\begin{proof}
\emph {The proof of case} 1. First, from lemma \ref{lem3.2}, we have $p_{-*}>p_+$ as $\rho_->\rho_m$. Then, from lemma \ref{lem3.1}, $p_{-*}$ is on $l(u_+,p_+): up^{\frac1{\gamma}}=u_+p_+^{\frac1{\gamma}}$. We conclude that $R_1(U,U_{-*})$ intersects with $R_3(U_+,U)$ at $U_2$ if it holds $\displaystyle u_{-*}+\frac{2}{\gamma-1}c_{-*}\ge u_{+}-\frac{2}{\gamma-1}c_{+}$, see Fig. 3.1. To prove that, it is enough to compare the relative positions of $R_1(u,p)$ and $l(u,p)$. Since one has
\begin{equation}\label{3.10}
\left\{
\begin{array}{l}
\displaystyle l(u,p):~~~~~~\frac{{\rm d}u}{{\rm d}p}\big|_{(u,p)}=-\frac{u}{\gamma p},\\[7pt]
\displaystyle R_1(u,p): ~\frac{{\rm d}u}{{\rm d}p}\big |_{(u,p)}=-\frac{1}{\rho c},
\end{array}
\right.
\end{equation}
it follows that 
\begin{equation}\label{3.11}
\displaystyle \frac{{\rm d}u}{{\rm d}p}\big |_{R_1(u,p)}-\frac{{\rm d}u}{{\rm d}p}\big|_{l(u,p)}=\frac{u}{\gamma p}-\frac{1}{\rho c}=\frac{u-c}{\rho c^2}>0.
\end{equation}
Thus $l(u,p)$ is above the curve $R_1(U,U_{-*})$ as $p<p_{-*}$. See Fig. 3.1(left). 

The interaction result in this case is: $U_-$ jumps to $U_{-*}$ by stationary wave, $U_{-*}$ reaches to $U_2$ by a backward rarefaction wave, followed by a contact discontinuity from $U_2$ to $\overline{U}_2$, then followed by a forward rarefaction wave from $\overline{U}_2$ to $U_+$. That is
\begin{equation}
J(U_m,U_-)\oplus S_0(U_+,U_m)\rightarrow S_0(U_{-*},U_-) \oplus R_1(U_2,U_{-*})  \oplus J(\overline{U}_2,U_2) \oplus R_3(U_+,\overline{U}_2).
\end{equation}

 If $\displaystyle u_{-*}+\frac{2}{\gamma-1}c_{-*}< u_{+}-\frac{2}{\gamma-1}c_{+}$,  then $R_1(U,U_{-*})\cap R_3(U_+,U)$ =$\emptyset$. $U_2$ turns to a vacuum, so as $\overline{U}_2$. The result is 
\begin{equation}
J(U_m,U_-)\oplus S_0(U_+,U_m)\rightarrow S_0(U_{-*},U_-) \oplus R_1({\rm Vacuum},U_{-*})  \oplus  R_3(U_+,{\rm Vacuum}).
\end{equation}

\begin{figure}[htbp]
\subfigure{
\begin{minipage}[t]{0.4\textwidth}
\centering
\includegraphics[width=0.95\textwidth]{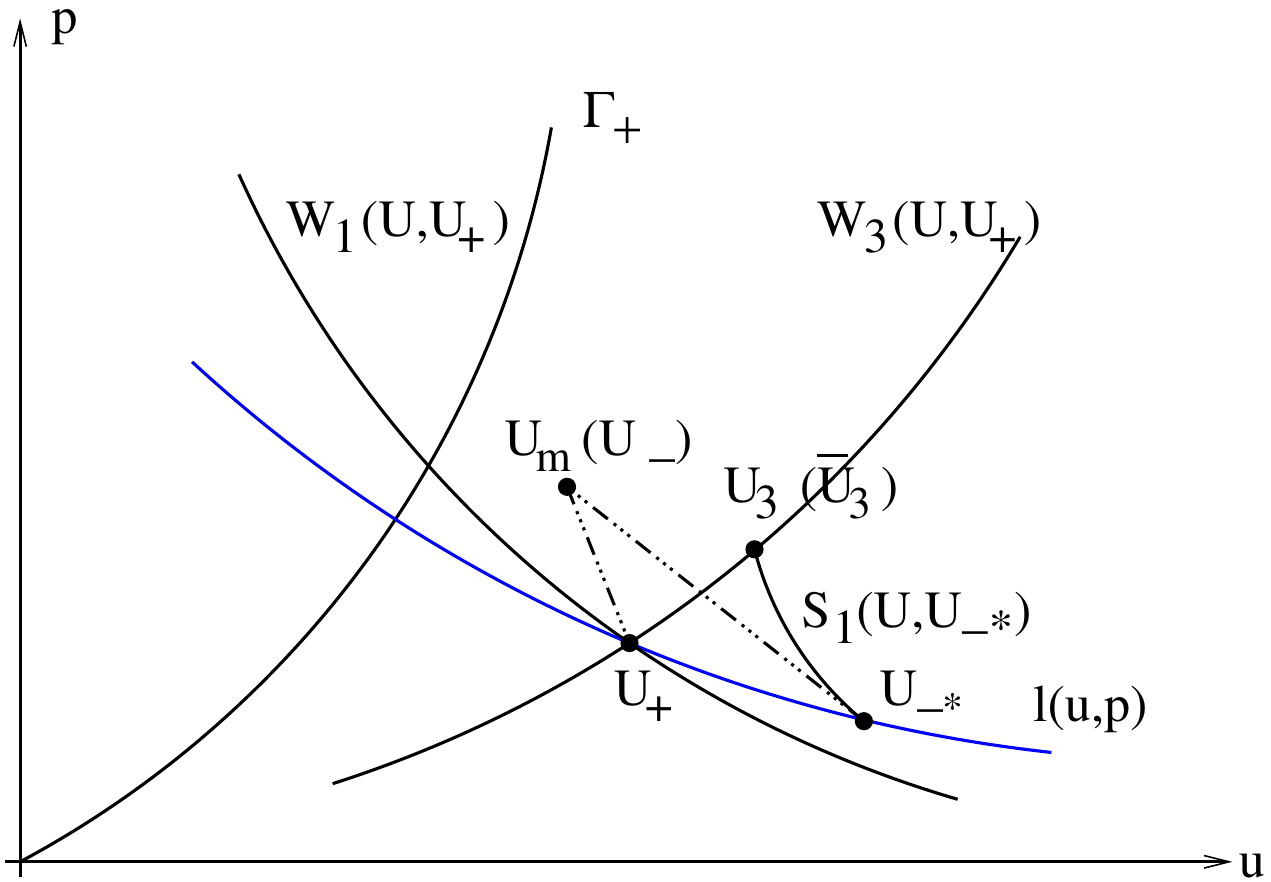}
\end{minipage}
}
\subfigure{
\begin{minipage}[t]{0.4\textwidth}
\centering
\includegraphics[width=0.95\textwidth]{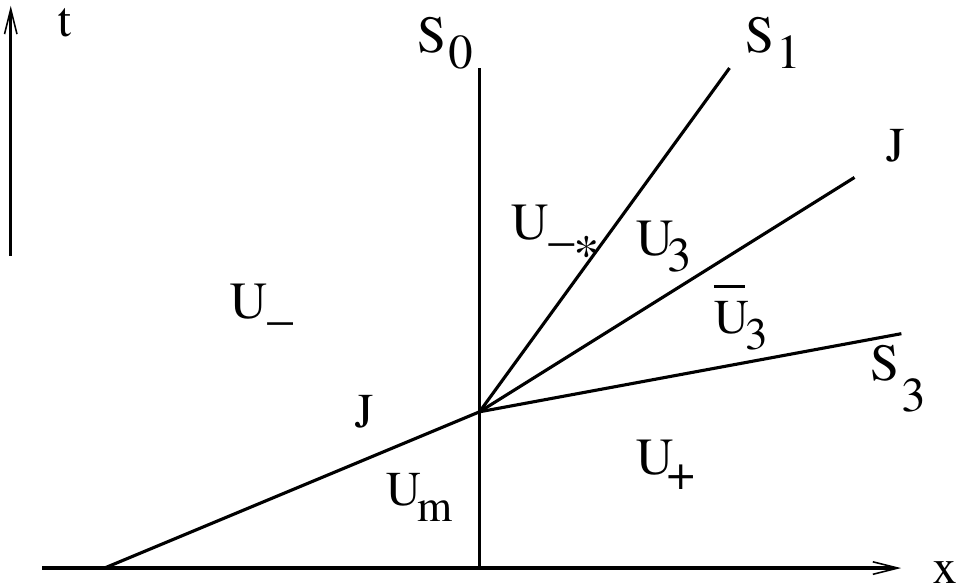}
\end{minipage}
}
\caption*{Fig. 3.2. Case 2. $u_m>c_m, u_->c_-$ and $\rho_-<\rho_m$.}
\end{figure}

\noindent
\emph {The proof of case} 2. First, from lemma \ref{lem3.2}, we have $p_{-*}<p_+$ as $\rho_-<\rho_m$, see Fig. 3.2. Denote that $\{U_3\}=S_1(U,U_{-*})\cap W_3(U,U_+)$, then we conclude that $p_3>p_+$
since $S_1(U,U_{-*})$ will not penetrate $W_1(U,U_+)$ in this case \cite{ChangHsiao3}. 

The interaction result is: $U_-$ jumps to $U_{-*}$ by stationary wave, $U_{-*}$ and $U_3$ are connected by a backward shock wave, followed by a contact discontinuity from $U_3$ to $\overline{U}_3$,  then followed by a forward shock wave from $\overline{U}_3$ to $U_+$. That is
\begin{equation}
J(U_m,U_-)\oplus S_0(U_+,U_m)\rightarrow S_0(U_{-*},U_-) \oplus S_1(U_3,U_{-*})  \oplus J(\overline{U}_3,U_3) \oplus S_3(U_+,\overline{U}_3).
\end{equation}
\end{proof}
\noindent
{\bf Remark 1.}   It is worthing to note that we consider the polytropic gas 
\eqref{2.1} here. For more general equations of state, such as the Chaplygin gas or the van der Waals gas, if $S_1(U,U_{-*})\cap W_3(U,U_+)=\emptyset$ in case 2, then a delta shock wave solution is needed. We left it for the future considerations.

{\bf Construction 2.  $u_m>c_m$ and $u_-<c_-$}.  \space This is a transonic case. The interaction results are obtained by solving a new Riemann problem with the initial data $U_-$ and $U_+$ once the contact discontinuity touches the stationary wave. For the details, we refer to \cite{LeflochThanh12,Thanh17}. The solution begins with a backward rarefaction wave from $U_-$ to a sonic point $U_C\in \Gamma_+$, followed by a stationary jump from $U_C$ to $U_{C*}$, then followed by a backward wave $W_1(U,U_{C*})$ from $U_{C*}$ to $U_4$, $U_4$ jumps to $\overline{U}_4$ by a contact discontinuity,  finally followed by a forward wave from  $\overline{U}_4$ to $U_+$. See Fig.3.3. Similarly, there exists a vacuum when 
$W_1(U,U_{C*})\cap W_3(U_+,U)=\emptyset$ in this case.

\begin{figure}[htbp]
\subfigure{
\begin{minipage}[t]{0.4\textwidth}
\centering
\includegraphics[width=0.95\textwidth]{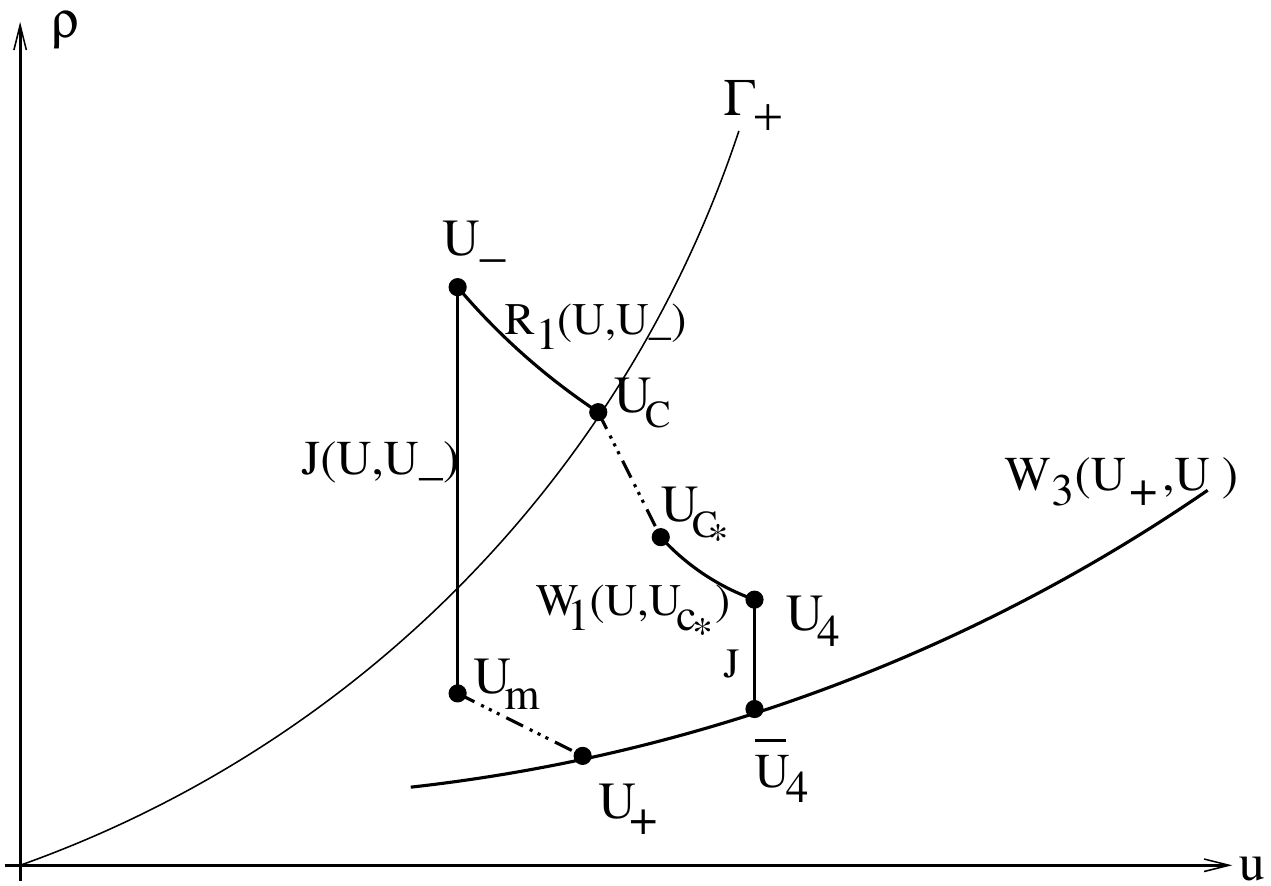}
\end{minipage}
}
\subfigure{
\begin{minipage}[t]{0.4\textwidth}
\centering
\includegraphics[width=0.95\textwidth]{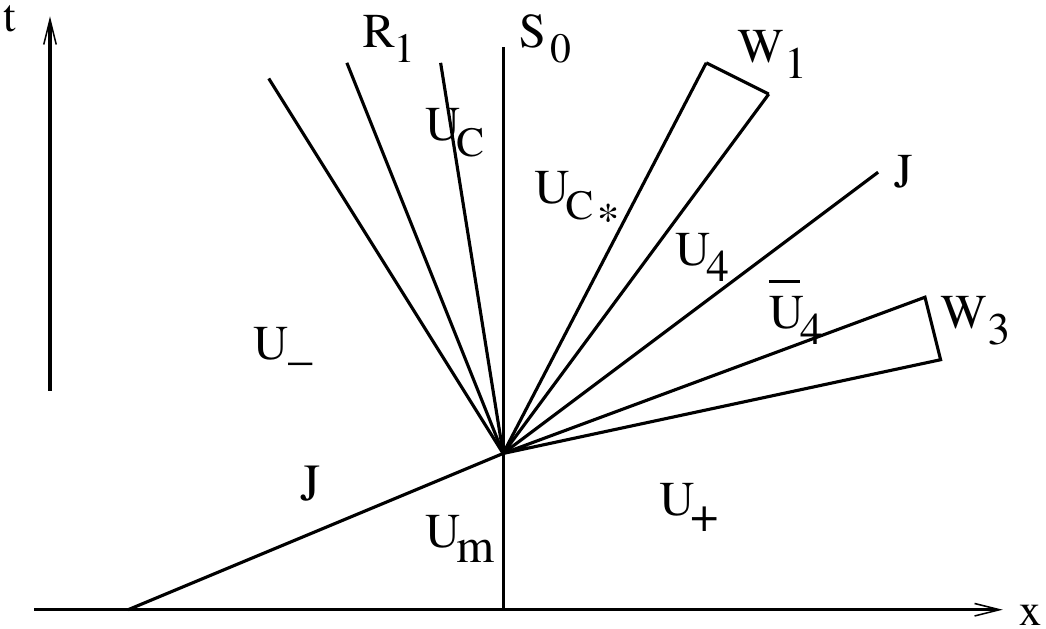}
\end{minipage}
}
\caption*{Fig. 3.3. Case $u_m>c_m$ and $u_-<c_-$.}
\end{figure}

{\bf Construction 3.  $u_m<c_m$ and $u_-<c_-$}.  \space  Now we turn to the case that $U_m$ is subsonic. First we consider that the left-hand state $U_-$ is also subsonic.  As the contact discontinuity touches the stationary wave, $U_-$ will first pass through a backward wave $W_1(U,U_-)$, which is different from the supersonic case. This indicates us to define a curve $ S_0(\overline{U}^{*},\overline{U})$ in the $(u,p)$ plane:
\begin{equation}\label{3.12}
 S_0(\overline{U}^{*},\overline{U}):  \quad \overline{U} \in W_1(U,U_-), \quad \overline{U}^{*} \in S_0(U,\overline{U}).
\end{equation}
It is obviously that $S_0(\overline{U}^{*},\overline{U})$ starts from $\overline{U}^{*}=U_{-}^{*}$ as $\overline{U}=U_-$. By using \eqref{3.9}, one
can discuss similarly as lemma \ref{lem3.2} to obtain that 
\begin{equation}
p_{-}^{*}>p_+ ~~ {\rm if} ~~ \rho_->\rho_{m}  \quad {\rm and} \quad 
p_{-}^{*}<p_+ ~~ {\rm if} ~~ \rho_-<\rho_{m}.
\end{equation}
In fact, from the stationary wave solution, one has $p_1>p_0=p_m$ in this case, which further indicates $\rho_1>\rho_0$ as $\displaystyle \frac{p_0}{\rho_0^{\gamma}}=\frac{p_1}{\rho_1^{\gamma}}$. Thus it follows that $\displaystyle \frac{{\rm d}p_1}{{\rm d}\rho_0}>0$ from \eqref{3.9}. 

Based on the relative positions of $\rho_-$ and $\rho_m$, we discuss the interaction results as follows.

\begin{lem}\label{lem3.4}
When $U_m$ and $U_-$ are both subsonic, $U_-$ first pass through a backward wave as the contact discontinuity touches the stationary wave. More specifically:

\noindent
Case 3.  $\rho_->\rho_m$, the interaction result is (see Fig. 3.4.):
\begin{equation}
J(U_m,U_-)\oplus S_0(U_+,U_m)\rightarrow R_1(U_2,U_-) \oplus S_0(U_2^{*},U_2)  \oplus J(\overline{U}_2^{*},U_2^{*})  \oplus S_3(U_+,\overline{U}_2^{*}).
\end{equation}
Case 4.  $\rho_-<\rho_m$, the interaction result as $1< \gamma\leq 2$ is (see Fig. 3.5.):
\begin{equation}
J(U_m,U_-)\oplus S_0(U_+,U_m)\rightarrow S_1(U_3,U_-) \oplus S_0(U_3^{*},U_3) \oplus J(\overline{U}_3^{*},U_3^{*}) \oplus W_3(U_+,\overline{U}_3^{*}).
\end{equation}
\end{lem}

\begin{figure}[h]
\subfigure{
\begin{minipage}[t]{0.4\textwidth}
\centering
\includegraphics[width=0.95\textwidth]{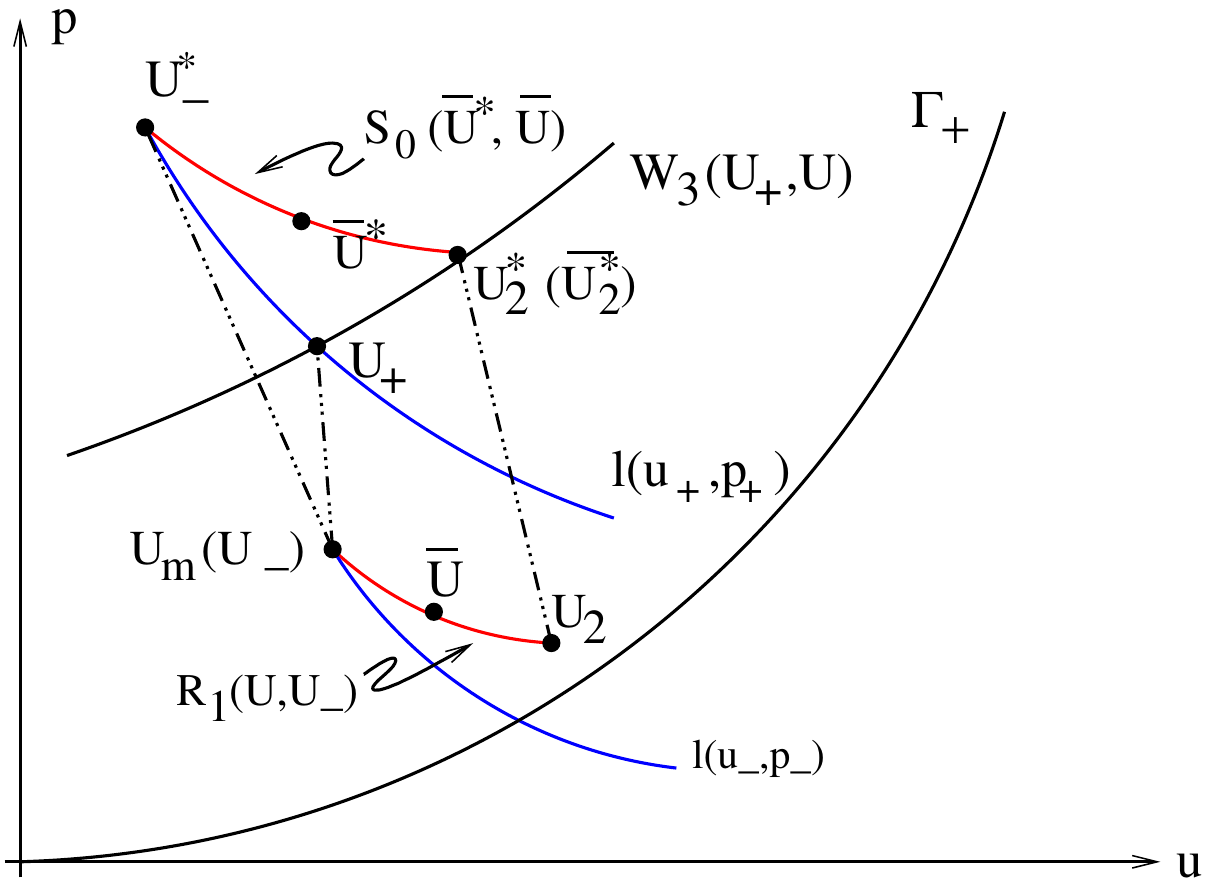}
\end{minipage}
}
\subfigure{
\begin{minipage}[t]{0.4\textwidth}
\centering
\includegraphics[width=0.95\textwidth]{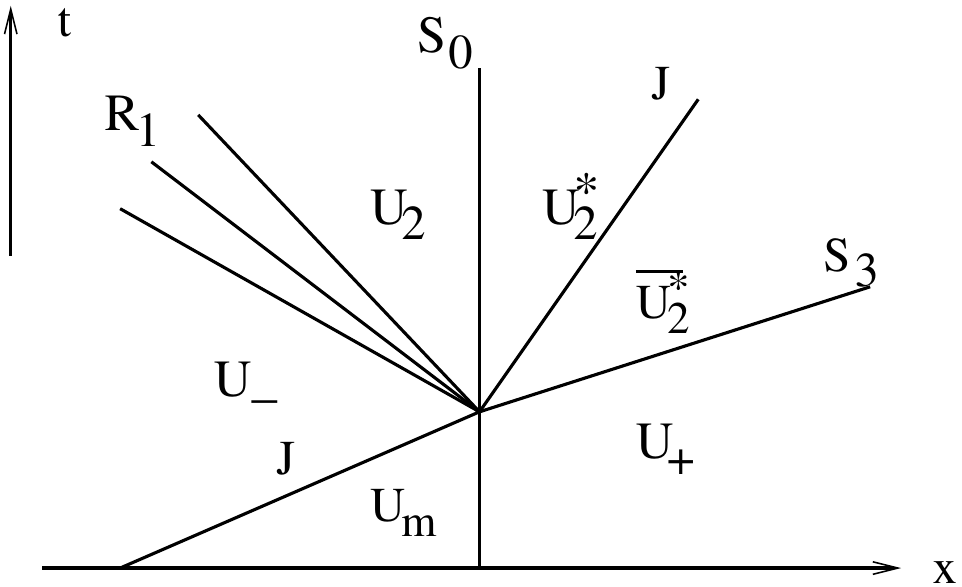}
\end{minipage}
}
\caption*{Fig. 3.4. Case 3. $u_m<c_m, u_-<c_-$ and $\rho_->\rho_m$.}
\end{figure}
\noindent
\begin{proof}
\emph {The proof of case} 3. On one hand, we have $p_{-}^{*}>p_+$ as $\rho_->\rho_m$ from the above discussion.  Note that $p_{-}^{*}$ is on the curve  $l(u_+,p_+): up^{\frac1{\gamma}}=u_+p_+^{\frac1{\gamma}}$ from lemma \ref{lem3.1}, see Fig. 3.4 (left). The interaction result starts from a backward rarefaction wave. It can be shown that $R_1(U,U_-)$ is above the curve $l(u_-,p_-): up^{\frac1{\gamma}}=u_-p_-^{\frac1{\gamma}}$ as $p<p_-$, which can be directly obtained from \eqref{3.11}.

On the other hand, to determine the forward wave, denote $\{U_2^{*}\}=S_0(\overline{U}^{*},\overline{U})\cap W_3(U_+,U)$, where $S_0(\overline{U}^{*},\overline{U})$ is defined in \eqref{3.12}. We next show that $U_2^{*}$ is above the curve $l(u_+,p_+)$. This is not difficult since from \eqref{3.3}, one has the following
\begin{equation}\label{3.14}
a_1u_2^{*}(p_2^{*})^{\frac1{\gamma}}=a_0 u_2p_2^{\frac1{\gamma}}>a_0u_-p_-^{\frac1{\gamma}}=a_1u_{-}^{*}(p_-^{*})^{\frac1{\gamma}}=a_1u_+p_+^{\frac1{\gamma}}.
\end{equation}
Here we use the fact 
\begin{equation}\label{3.13}
u_2p_2^{\frac1{\gamma}}>u_-p_-^{\frac1{\gamma}}
\end{equation}
from the above discussion. Thus the forward wave can be determined.

The interaction result in this case is: $U_-$ first reaches to $U_2$ by a backward rarefaction wave, followed by a stationary wave from $U_2$ to $U_2^{*}$, followed by a contact discontinuity from $U_2^{*}$ to $\overline{U}_2^{*}$, then followed by a forward shock wave from $\overline{U}_2^{*}$ to $U_+$. That is 
\begin{equation}
J(U_m,U_-)\oplus S_0(U_+,U_m)\rightarrow R_1(U_2,U_-) \oplus S_0(U_2^{*},U_2)  \oplus J(\overline{U}_2^{*},U_2^{*})  \oplus S_3(U_+,\overline{U}_2^{*}).
\end{equation}

\begin{figure}[htbp]
\subfigure{
\begin{minipage}[t]{0.4\textwidth}
\centering
\includegraphics[width=0.95\textwidth]{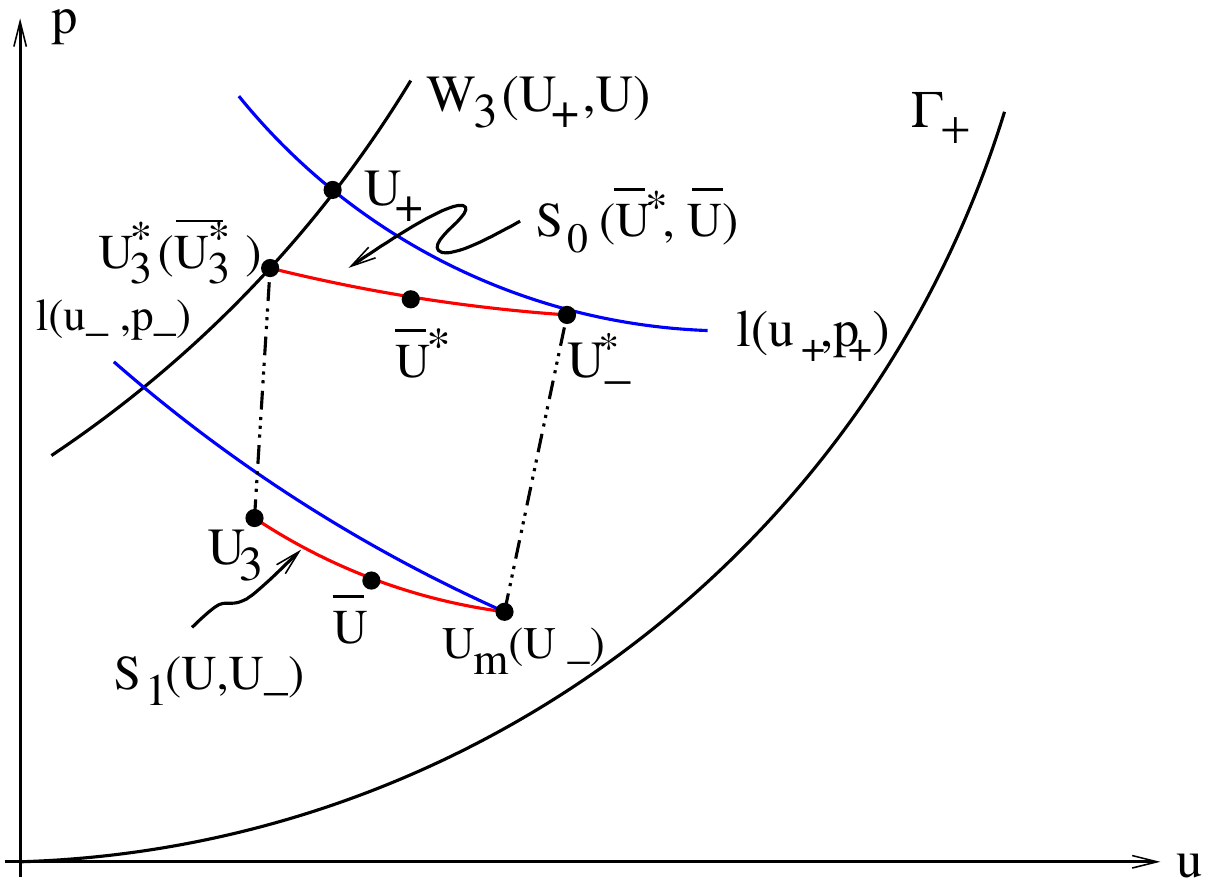}
\end{minipage}
}
\subfigure{
\begin{minipage}[t]{0.4\textwidth}
\centering
\includegraphics[width=0.95\textwidth]{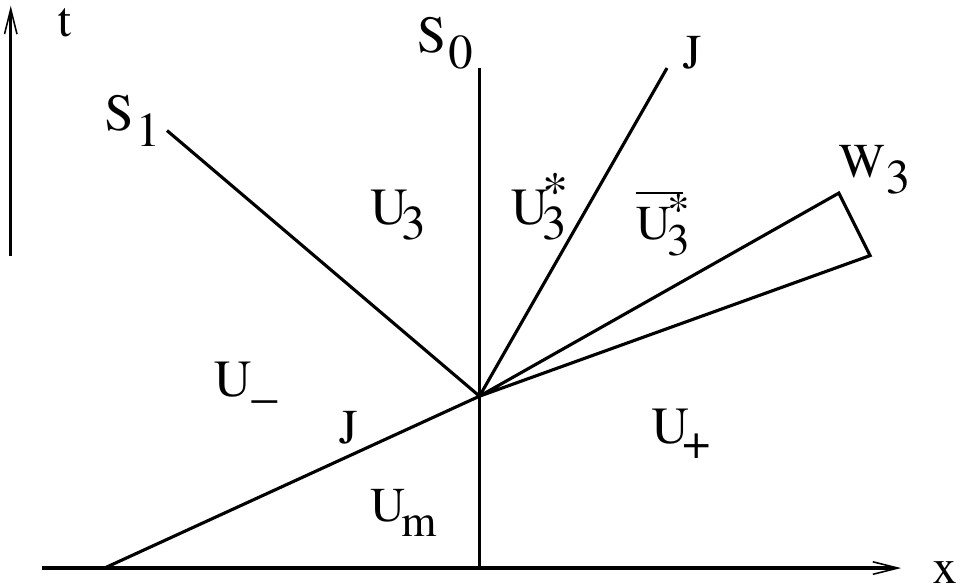}
\end{minipage}
}
\caption*{Fig. 3.5. Case 4. $u_m<c_m, u_-<c_-$ and $\rho_-<\rho_m$.}
\end{figure}

\noindent
 \emph {The proof of case} 4. On one hand, we have $p_{-}^{*}<p_+$ as $\rho_-<\rho_m$  from \eqref{3.9}. Thus the interaction result starts from a backward shock wave first, see Fig. 3.5 (left). On the other hand, to determine the forward wave, denote $\{U_3^{*}\}=S_0(\overline{U}^{*},\overline{U})\cap W_3(U_+,U)$, where $S_0(\overline{U}^{*},\overline{U})$ is defined in \eqref{3.12}. $U_3^{*}$ is jumped from $U_3$ by a stationary wave. It is sufficient to show that $U_3$ is on the left side of $l(u_-,p_-)$ as $1<\gamma\leq2$. The conclusion is not obviously since direct comparison of the positions between $S_1(U,U_-)$ and $l(u_-,p_-)$ may bring difficulties. Here we use the curve $R_1(U,U_-)$ to prove it. To this end, from \eqref{2.18}, one has
\begin{equation}\label{3.15}
\displaystyle \frac{{\rm d} u}{{\rm d} p}\big|_{S_1(U,U_-)}=-\sqrt{\frac{(1-\mu^2)\tau}{\mu^2 p+p_-}}\frac{p+(1+2\mu^2)p_-}{2(p+\mu^2p_-)}.
\end{equation}
Since $S_1(U,U_-)$ touches $R_1(U,U_-)$ up to the second order (\cite{Smoller}). It is necessary to show that 
\begin{equation}\label{3.16}
\displaystyle \frac{{\rm d} u}{{\rm d} p}\big|_{S_1(U,U_-)}- \frac{{\rm d} u}{{\rm d} p}\big|_{R_1(U,U_-)}<0 
\end{equation}
as $p>p_-$ when $1<\gamma \leq 2$. One can compute \eqref{3.16} by
\begin{equation}\label{3.17}
\begin{array}{lll}
\displaystyle \frac{{\rm d} u}{{\rm d} p}\big|_{S_1(U,U_-)}- \frac{{\rm d} u}{{\rm d} p}\big|_{R_1(U,U_-)}&=\displaystyle \sqrt{\frac{\tau}{ \gamma p}}-\sqrt{\frac{(1-\mu^2)\tau}{\mu^2 p+p_-}}\frac{p+(1+2\mu^2)p_-}{2(p+\mu^2p_-)}\\[12pt]
 &= \displaystyle \sqrt{\frac{(1-\mu^2)\tau}{\mu^2p+p_0}}\left(\sqrt{\frac{\mu^2p+p_0}{\gamma(1-\mu^2)p}}-\frac{p+(1+2\mu^2)p_0}{2(p+\mu^2p_0)}\right).
\end{array}
\end{equation}
We are left to determine the sign of
\begin{equation}\label{3.18}
\displaystyle \sqrt{\frac{\mu^2p+p_0}{\gamma(1-\mu^2)p}}-\frac{p+(1+2\mu^2)p_0}{2(p+\mu^2p_0)}.
\end{equation}
Set 
\begin{equation}\label{3.19}
\displaystyle  f(p,p_0): =\frac{\mu^2p+p_0}{\gamma(1-\mu^2)p}-\left(\frac{p+(1+2\mu^2)p_0}{2(p+\mu^2p_0)}\right)^2,
\end{equation}
if we make $x=p/p_0, x>1$, then a direct calculation shows that the size of \eqref{3.19} is equivalent to 
\begin{align}\label{3.20}
\displaystyle  g(x): &=(4\mu^2-\gamma(1-\mu^2))x^3+(4+8\mu^4-2\gamma(1-\mu^2)(1+2\mu^2))x^2\\
&+(8\mu^2-\gamma(1-\mu^2)(1+2\mu^2)^2+4\mu^6)x+4\mu^4.
\end{align}
One can easily verify that $g(1)=g'(1)=0$, $g'(x)<0$ for $x>1$ as $1<\gamma \leq 2$. 
Thus the curve $S_1(U,U_-)$ is below $R_1(U,U_-)$ as $p>p_-$ when $1<\gamma \leq 2$. Besides, from \eqref{3.11}, one already knows that $R_1(U,U_-)$ is always below $l(u_-,p_-)$ as $p>p_-$. This leads to the fact that $S_1(U,U_-)$ is below $l(u_-,p_-)$ as $p>p_-$ when $1<\gamma \leq 2$.

Similar as \eqref{3.14}, one may show that $p_3^{*}<p_+$, see Fig. 3.5. The interaction result in this case is: $U_-$ and $U_3$ are connected by a backward shock wave, followed by a stationary wave from $U_3$ to $U_3^{*}$, then followed by a forward rarefaction wave from $U_3^{*}$ to $U_+$. That is 
\begin{equation}
J(U_m,U_-)\oplus S_0(U_+,U_m)\rightarrow S_1(U_3,U_-) \oplus S_0(U_3^{*},U_3) \oplus R_3(U_+,U_3^{*}).
\end{equation}
\end{proof}

\begin{figure}[htbp]
\subfigure{
\begin{minipage}[t]{0.31\textwidth}
\centering
\includegraphics[width=0.985\textwidth]{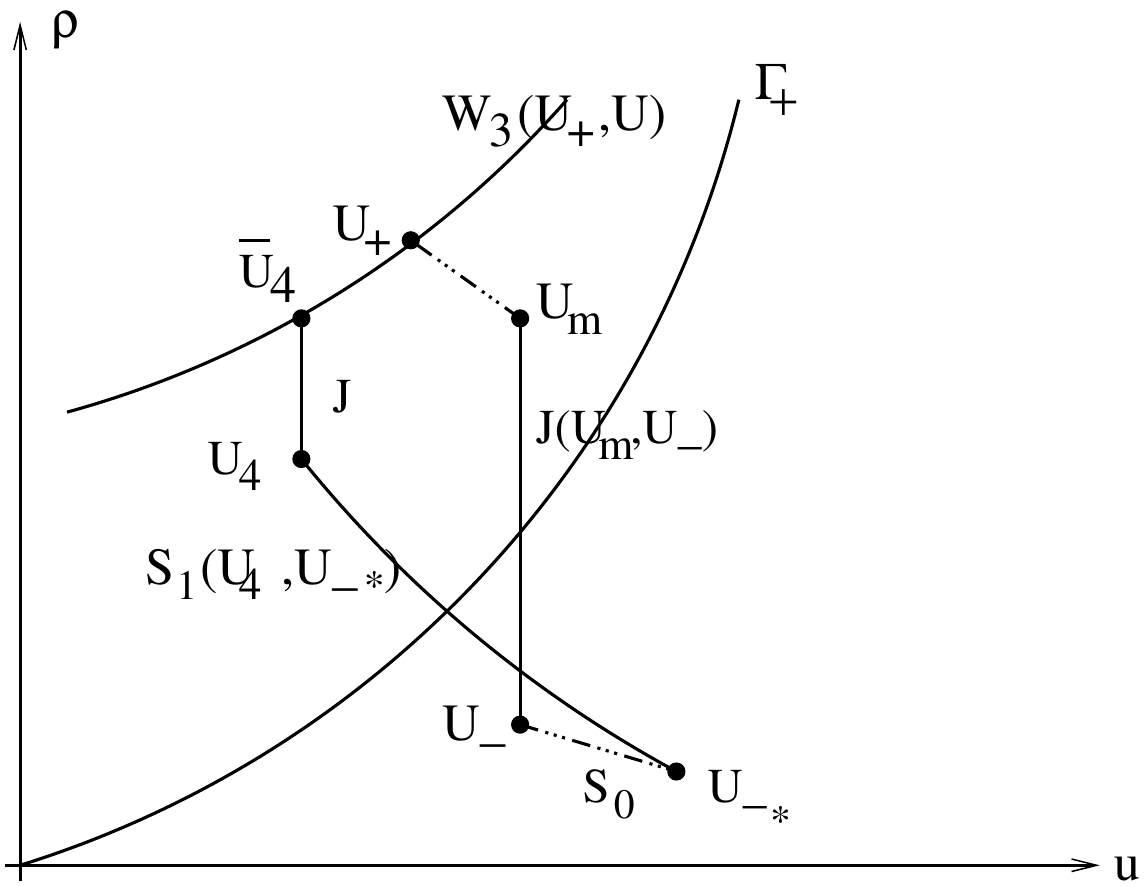}
\end{minipage}
}
\subfigure{
\begin{minipage}[t]{0.31\textwidth}
\centering
\includegraphics[width=0.985\textwidth]{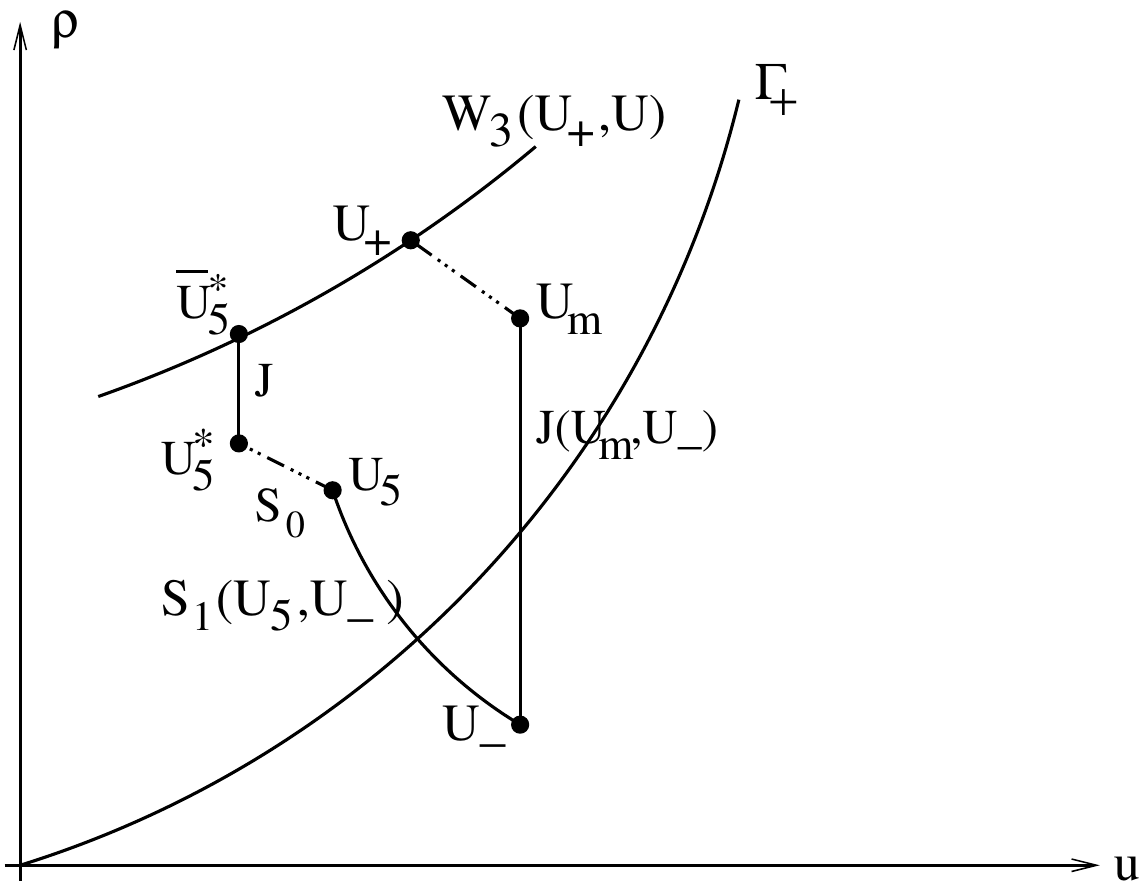}
\end{minipage}
}
\subfigure{
\begin{minipage}[t]{0.31\textwidth}
\centering
\includegraphics[width=0.985\textwidth]{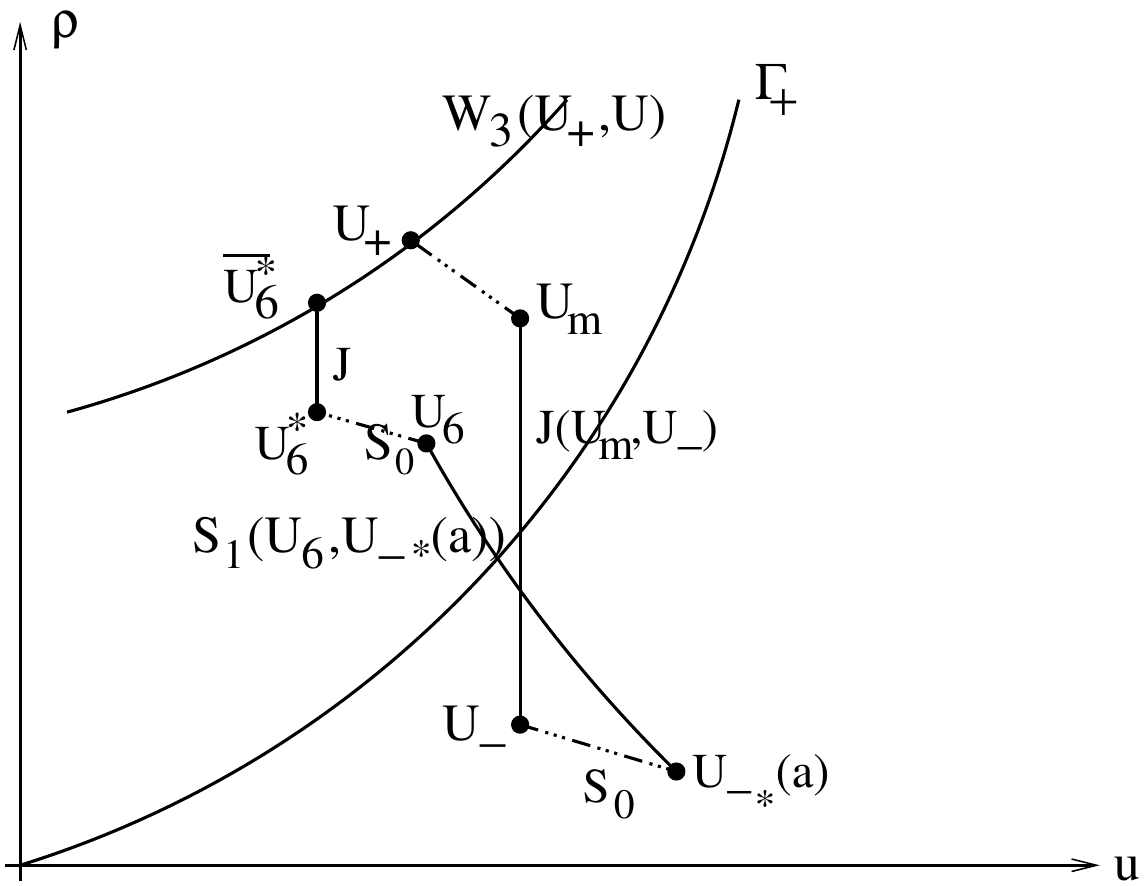}
\end{minipage}
}
\caption*{Fig. 3.6. Case $u_m<c_m$ and $u_->c_-$.}
\end{figure}

{\bf Construction 4.  $u_m<c_m$ and $u_->c_-$}. \space  We are left with the case $u_->c_-$ as $u_m<c_m$. This is also a transonic case. The new Riemann problem as the interaction happens have at most three solutions. See Fig. 3.6. For the first solution, $U_-$ jumps to $U_{-*}$ by a stationary wave, followed by a backward shock wave with positive speed, then followed by a forward wave. That is
\begin{equation}
J(U_m,U_-)\oplus S_0(U_+,U_m)\rightarrow S_0(U_{-*},U_-) \oplus S_1(U_4,U_{-*}) \oplus J(\overline{U}_4,U_4) \oplus W_3(U_+,\overline{U}_4).
\end{equation}
 For the second solution, $U_-$ jumps to a subsonic state by a backward shock wave, followed by a stationary wave, then followed by a forward wave. That is
\begin{equation}
J(U_m,U_-)\oplus S_0(U_+,U_m)\rightarrow S_1(U_5,U_-) \oplus S_0(U_5^{*},U_5)\oplus J(\overline{U}_5^{*},U_5^{*}) \oplus W_3(U_+,\overline{U}_5^{*}).
\end{equation}
For the third solution, it contains three waves with the same zero speed. That is
\begin{equation}
 S_0(U_{-*}(a),U_-) \oplus 
S_1(U_6,U_{-*} (a)) \oplus  S_0(U_6^{*},U_6) \oplus J(\overline{U}_6^{*},U_6^{*}) \oplus W_3(U_+,U_6^{*}).
\end{equation}
Here $U_{-*} (a)$ is jumped from $U_-$ by stationary wave with the cross section shifting from $a_0$ to an intermediate state $a\in[a_0,a_1]$. We refer to \cite{Thanh17} for more details.

\section{Numerical simulations}\label{sec4}
In this section we give some numerical examples, which is consistent with our analysis in section \ref{sec3}. Given a uniform time step $\Delta t$ and an equal mesh size $\Delta x$. Set $x_j=j\Delta x$, $j\in {\bf Z}$, $t^{n}=n\Delta{t}$, $n\in {\bf N}$. Set 
\begin{equation}\label{4.1}
\lambda=\frac{\Delta t}{\Delta x}.
\end{equation}

Let $V_j^{n}$ be the approximation of the values $V(x_j,t^{n})$ of the exact solution. Here we use the modified Godunov-Rusanov scheme (see \cite{SaurelAbgrall16})
\begin{equation}\label{4.2}
\begin{array}{ll}
\displaystyle V_i^{n+1}=V_i^{n}-\lambda\left(F_{i+1/2}^{n}-F_{i-1/2}^{n}\right)+\Delta t p_i^{n} \Delta,\\[10pt]
V :=(a\rho, a\rho u, a\rho E),\quad F(V) := (a\rho u,a(\rho u^2+p), au(\rho E+p)),
\end{array}
\end{equation}
where $\Delta$ represents the discrete form of the term $a_x$ which is set to 
\begin{equation}\label{4.3}
\displaystyle \Delta=\frac{a_{i+1}^{n}-a_i^{n}}{2\Delta x}.
\end{equation}
The numerical flux for the conservative fluxes is given by
\begin{equation}\label{4.4}
\displaystyle F_{i+1/2}^{n}=\frac12\left(F_i+F_{i+1}-S_{i+1/2}(V_{i+1}-V_i)\right),
\end{equation}
where $S_{i+1/2}={\rm max}\{|(\lambda_1)_i|, |(\lambda_3)_i|, |(\lambda_1)_{i+1}|, |(\lambda_3)_{i+1}|\}$.

The  domain is set to [0,10], the stationary wave is located at $x=3.0$ for clearly seen. We use 2000 grids, the CFL constant is set to 0.75, $\gamma=2.0$. 

\noindent 
{\bf Test 1.} The initial data is given by
\begin{equation}\label{4.5}
\begin{array}{lll}
\displaystyle (\rho_-,u_-,p_-,a_0)=(2.25,5.0, 5.0, 1.0),\quad 0<x<2.9,\\
 (\rho_m,u_m,p_m,a_0)=(1.0,5.0, 5.0, 1.0), \quad 2.9<x<3,\\
 (\rho_+,u_+,p_+,a_1)=(0.688168,5.589, 2.3679, 1.5), \quad 3<x<10.
\end{array}
\end{equation}

\begin{figure}[htbp]
\subfigure{
\begin{minipage}[t]{0.31\textwidth}
\centering
\includegraphics[width=\textwidth]{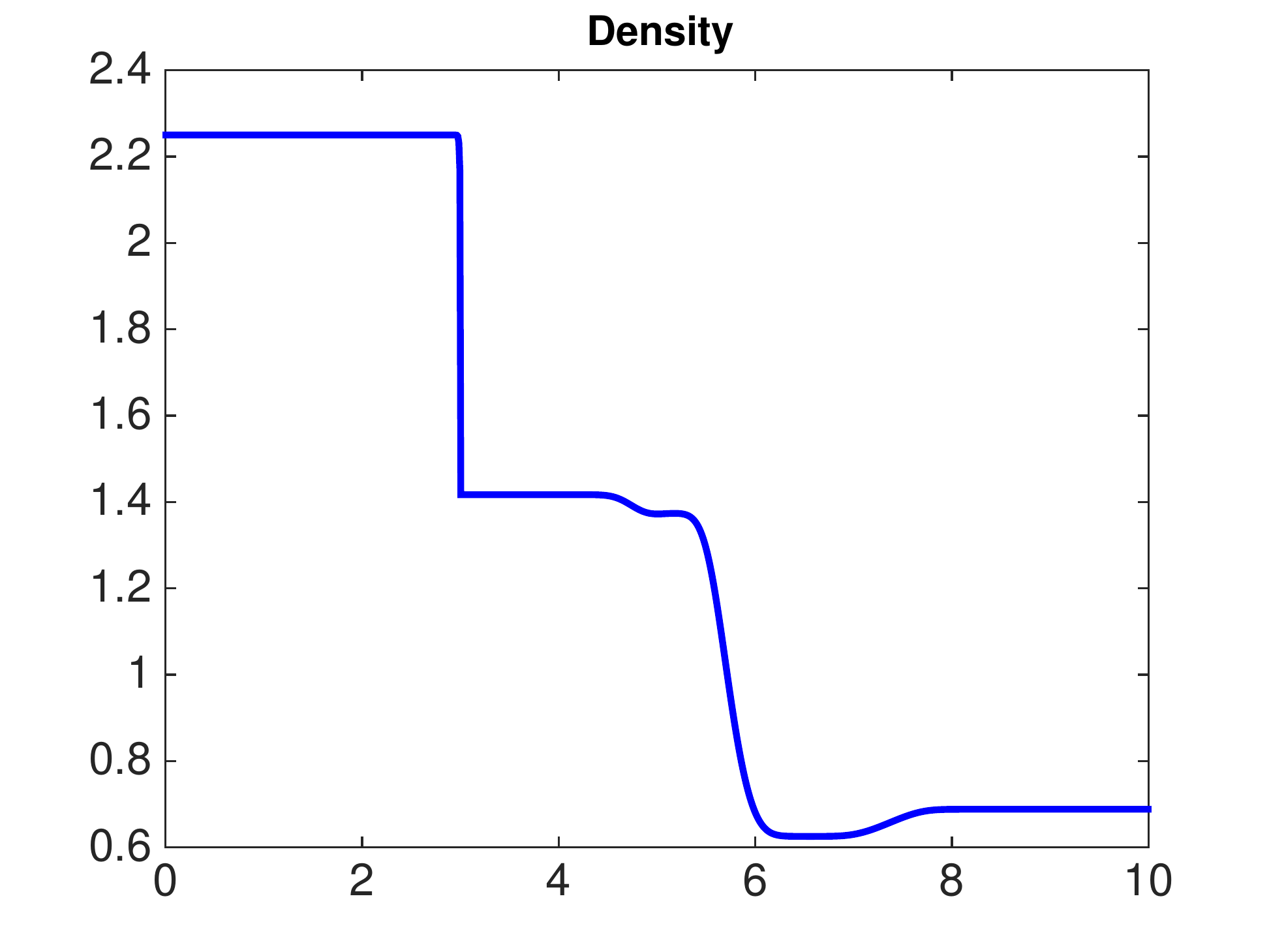}
\end{minipage}
}
\subfigure{
\begin{minipage}[t]{0.31\textwidth}
\centering
\includegraphics[width=\textwidth]{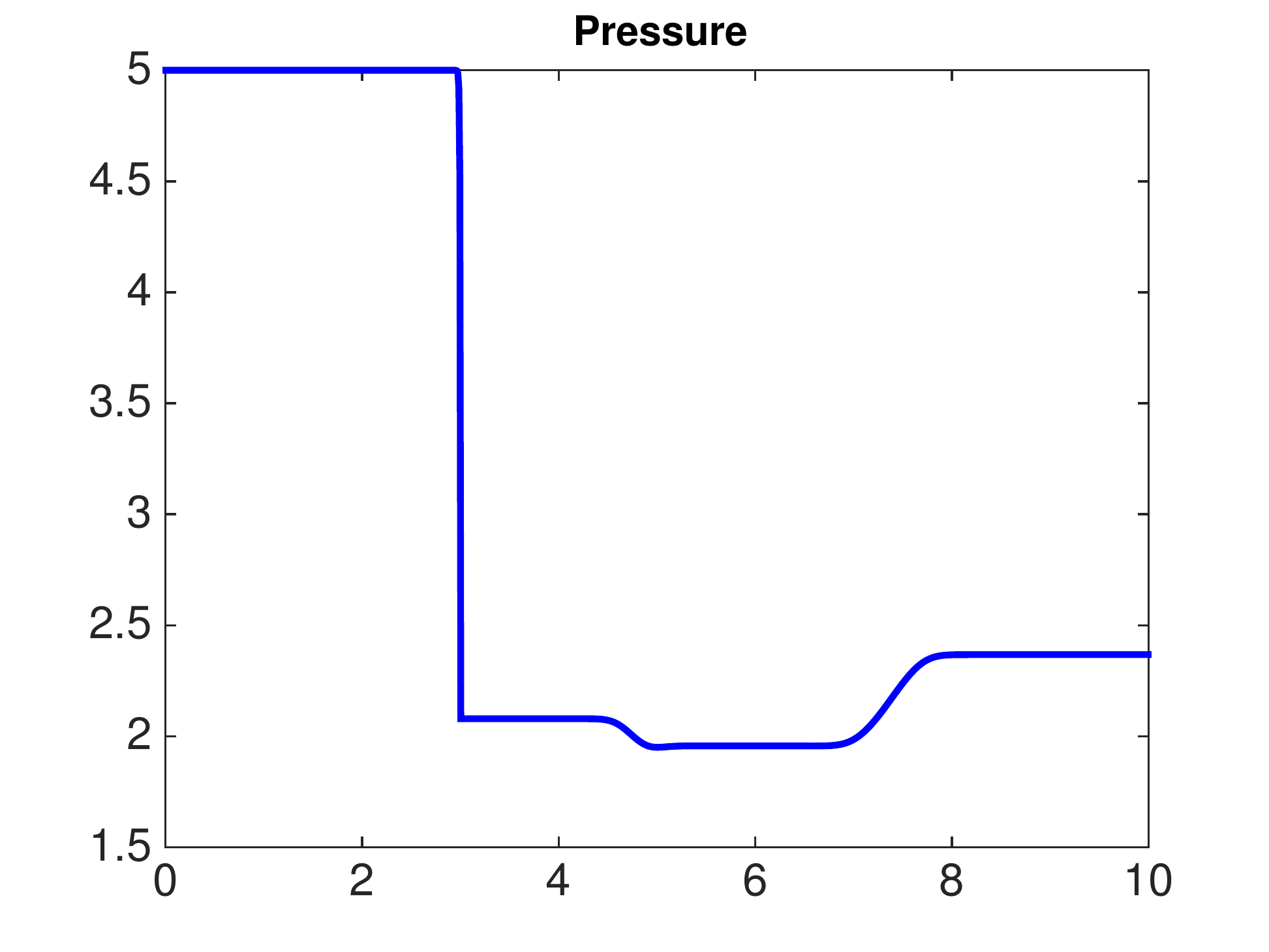}
\end{minipage}
}
\subfigure{
\begin{minipage}[t]{0.31\textwidth}
\centering
\includegraphics[width=\textwidth]{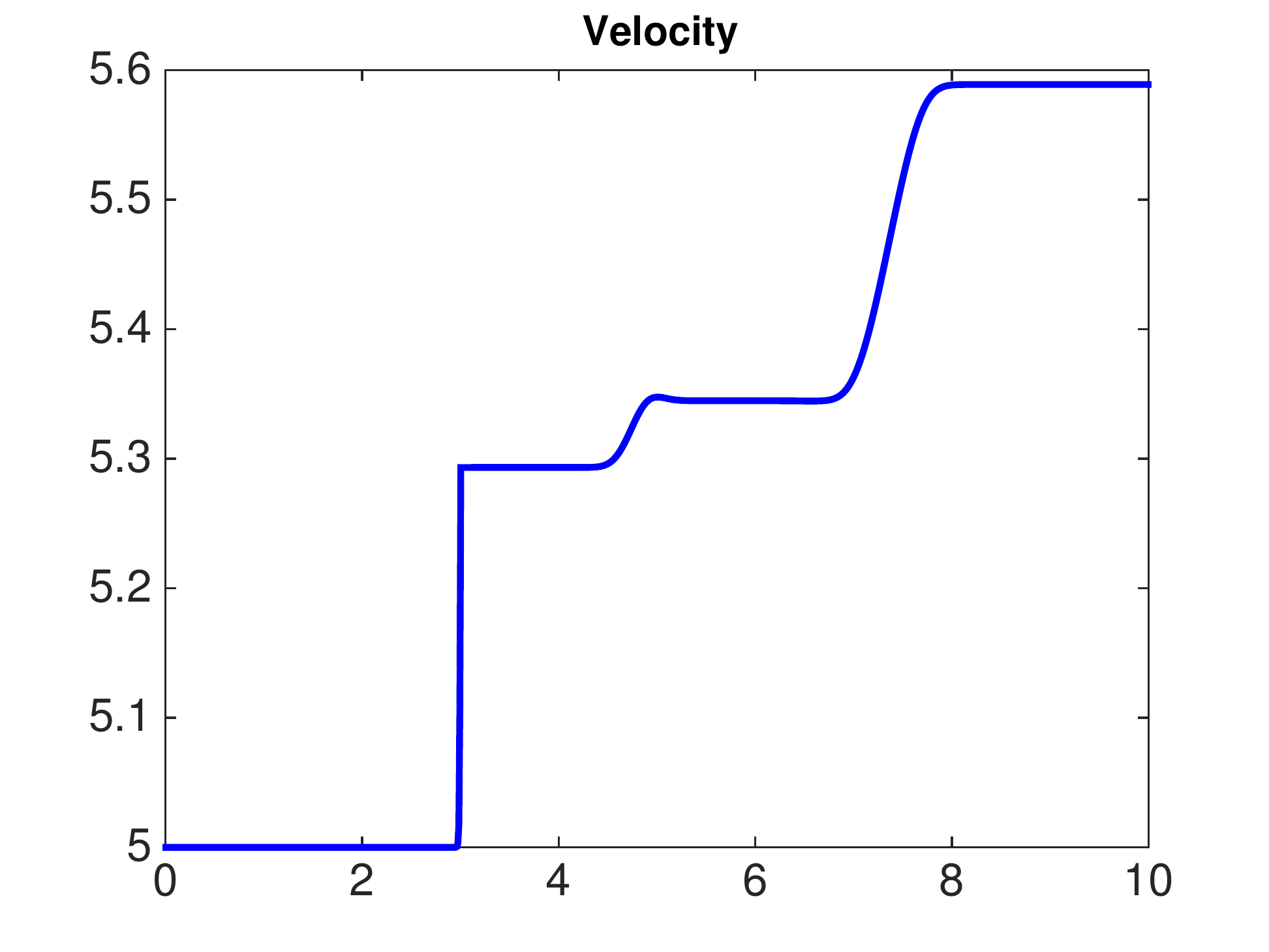}
\end{minipage}
}
\caption*{Fig. 4.1. Test 1.}
\end{figure}
We have $U_-\in D_1$, $U_m\in D_1$, $U_+\in D_1$. The result is shown at $t=0.35s$, see Fig. 4.1. The solution begins with a stationary wave, followed by a backward rarefaction wave, followed by a contact discontinuity, then followed by a forward rarefaction wave. The result is the same with that in case 1.

\noindent
{\bf Test 2.} The initial data is given by
\begin{equation}\label{4.6}
\begin{array}{lll}
\displaystyle (\rho_-,u_-,p_-,a_0)=(0.75, 5.0, 5.0, 1.0),\quad 0<x<2.9,\\
 (\rho_m,u_m,p_m,a_0)=(1.0, 5.0, 5.0, 1.0), \quad 2.9<x<3,\\
 (\rho_+,u_+,p_+,a_1)=(0.688168, 5.589, 2.3679, 1.3), \quad 3<x<10.
\end{array}
\end{equation}

\begin{figure}[htbp]
\subfigure{
\begin{minipage}[t]{0.31\textwidth}
\centering
\includegraphics[width=0.95\textwidth]{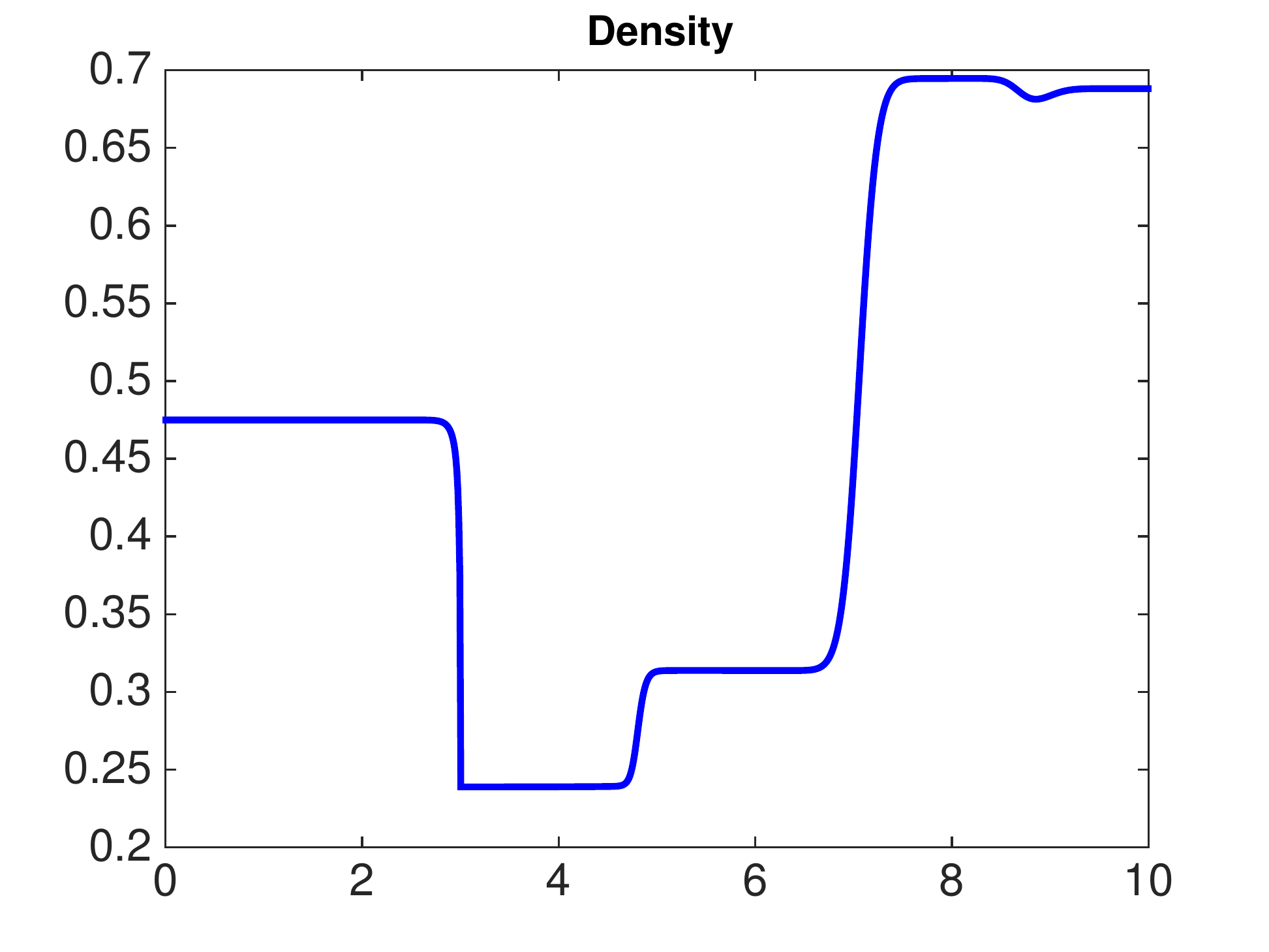}
\end{minipage}
}
\subfigure{
\begin{minipage}[t]{0.31\textwidth}
\centering
\includegraphics[width=0.95\textwidth]{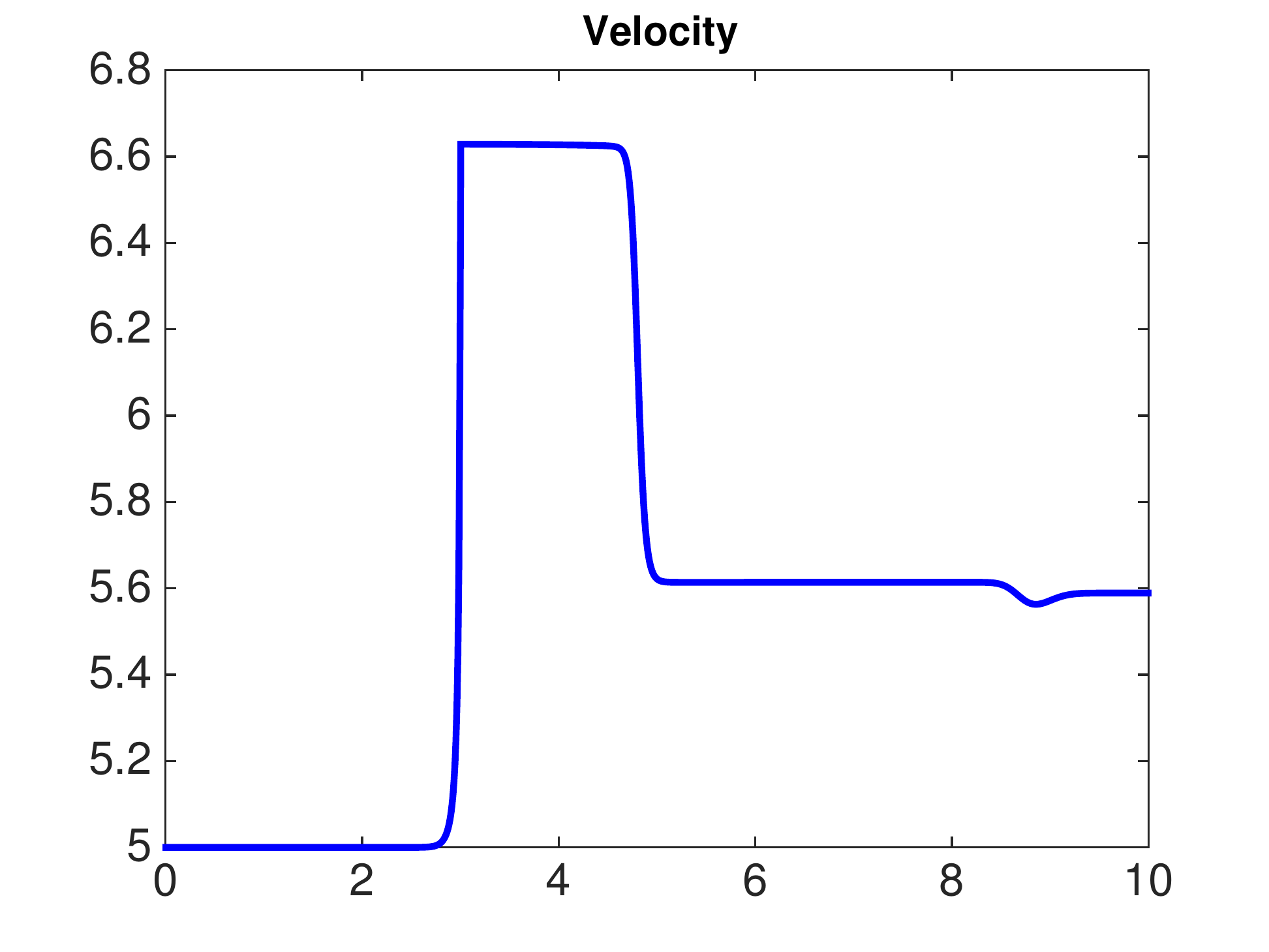}
\end{minipage}
}
\subfigure{
\begin{minipage}[t]{0.31\textwidth}
\centering
\includegraphics[width=0.95\textwidth]{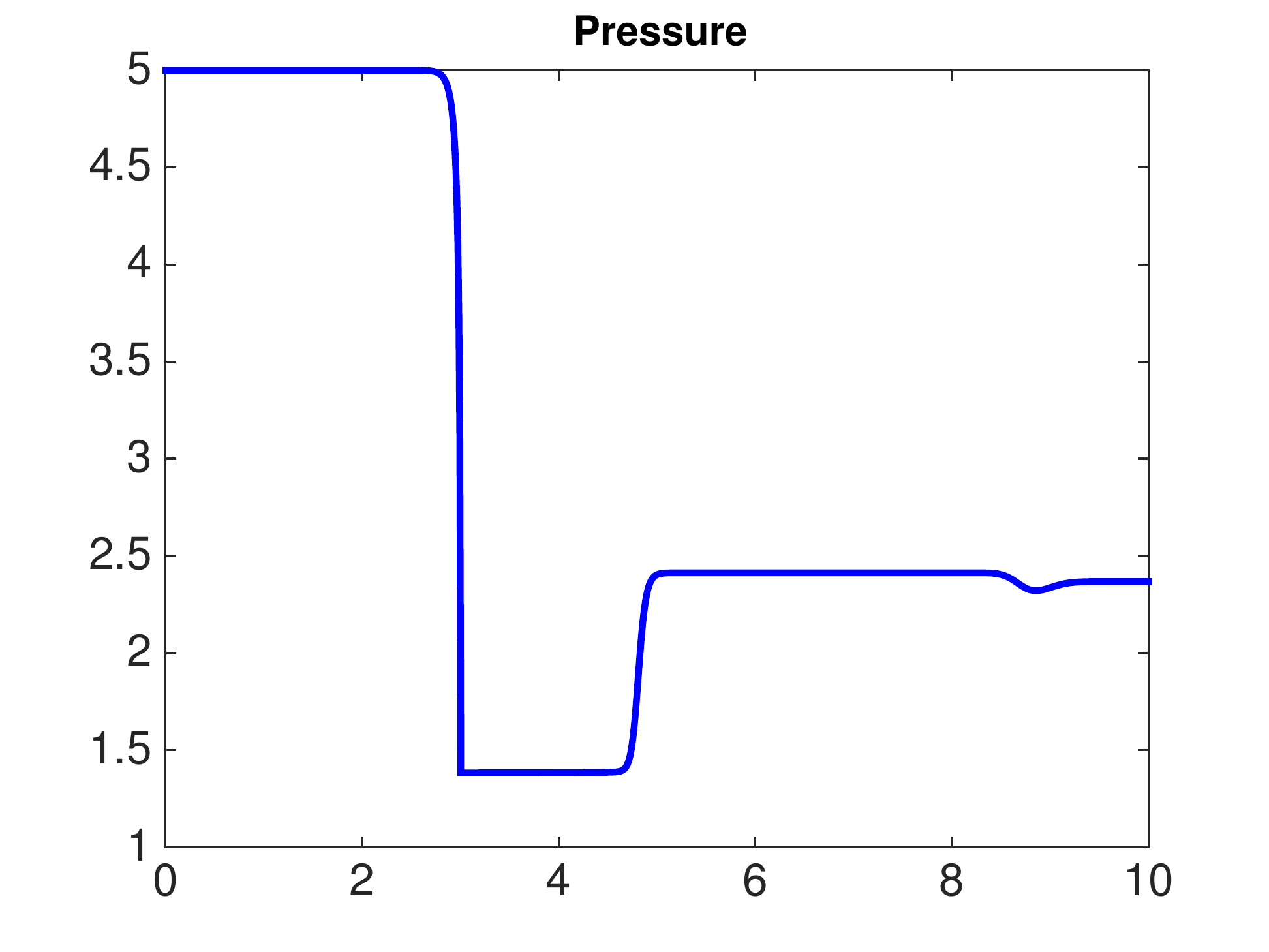}
\end{minipage}
}
\caption*{Fig. 4.2. Test 2.}
\end{figure}

We have $U_-\in D_1$, $U_m\in D_1$, $U_+\in D_1$. The result is shown at $t=0.35s$, see Fig. 4.2. The solution begins with a stationary wave, followed by a backward shock wave, followed by a contact discontinuity, then followed by a forward shock wave. The result is the same with that in case 2.

\noindent
{\bf Test 3.} The initial data is given by
\begin{equation}\label{4.7}
\begin{array}{lll}
\displaystyle (\rho_-,u_-,p_-,a_0)=( 0.25, 5.0,5.0, 1.0),\quad 0<x<2.9,\\
 (\rho_m,u_m,p_m,a_0)=(1.0, 5.0, 5.0, 1.0), \quad 2.9<x<3,\\
 (\rho_+,u_+,p_+,a_1)=(0.688168, 5.589, 2.3679, 1.5), \quad 3<x<10.
\end{array}
\end{equation}

We have $U_-\in D_2$, $U_m\in D_1$, $U_+\in D_1$. The result is shown at $t=0.5s$, see Fig. 4.3. The solution begins with a backward rarefaction wave, which is attached with the stationary wave, followed by a backward shock wave, followed by a contact discontinuity, then followed by a forward shock wave. The result is the same with that in the transonic case.

\begin{figure}[htbp]
\subfigure{
\begin{minipage}[t]{0.31\textwidth}
\centering
\includegraphics[width=0.95\textwidth]{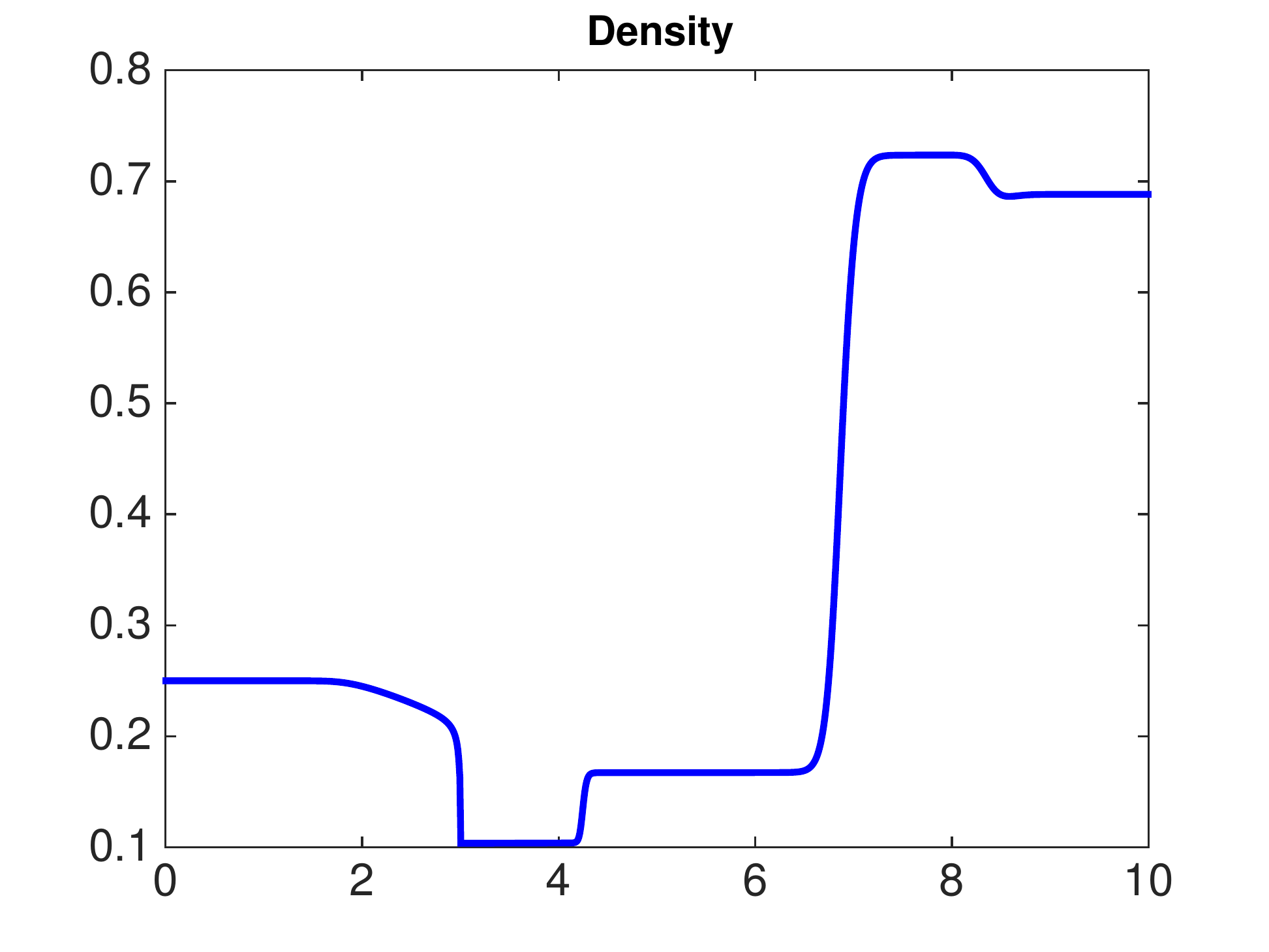}
\end{minipage}
}
\subfigure{
\begin{minipage}[t]{0.31\textwidth}
\centering
\includegraphics[width=0.95\textwidth]{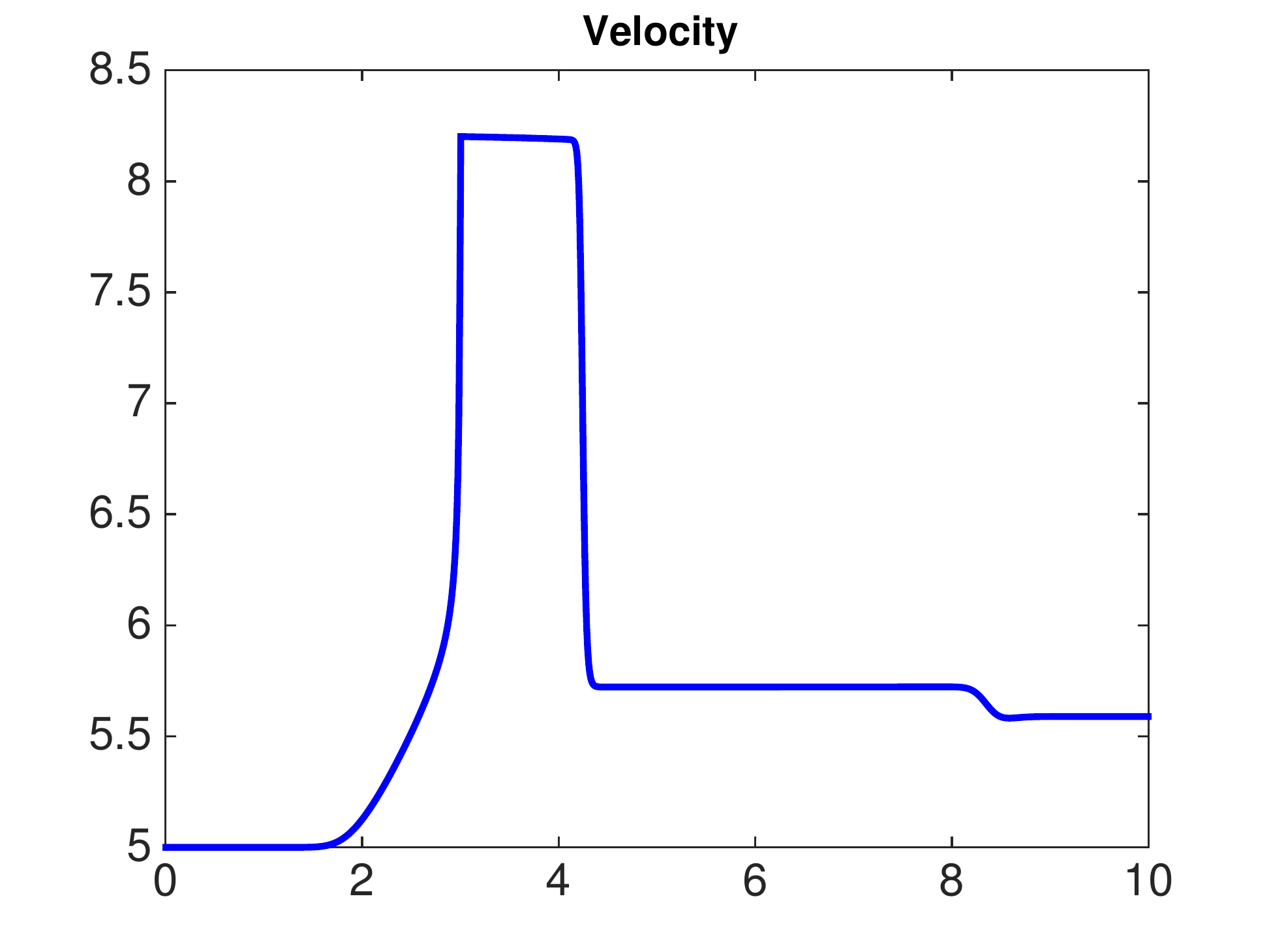}
\end{minipage}
}
\subfigure{
\begin{minipage}[t]{0.31\textwidth}
\centering
\includegraphics[width=0.95\textwidth]{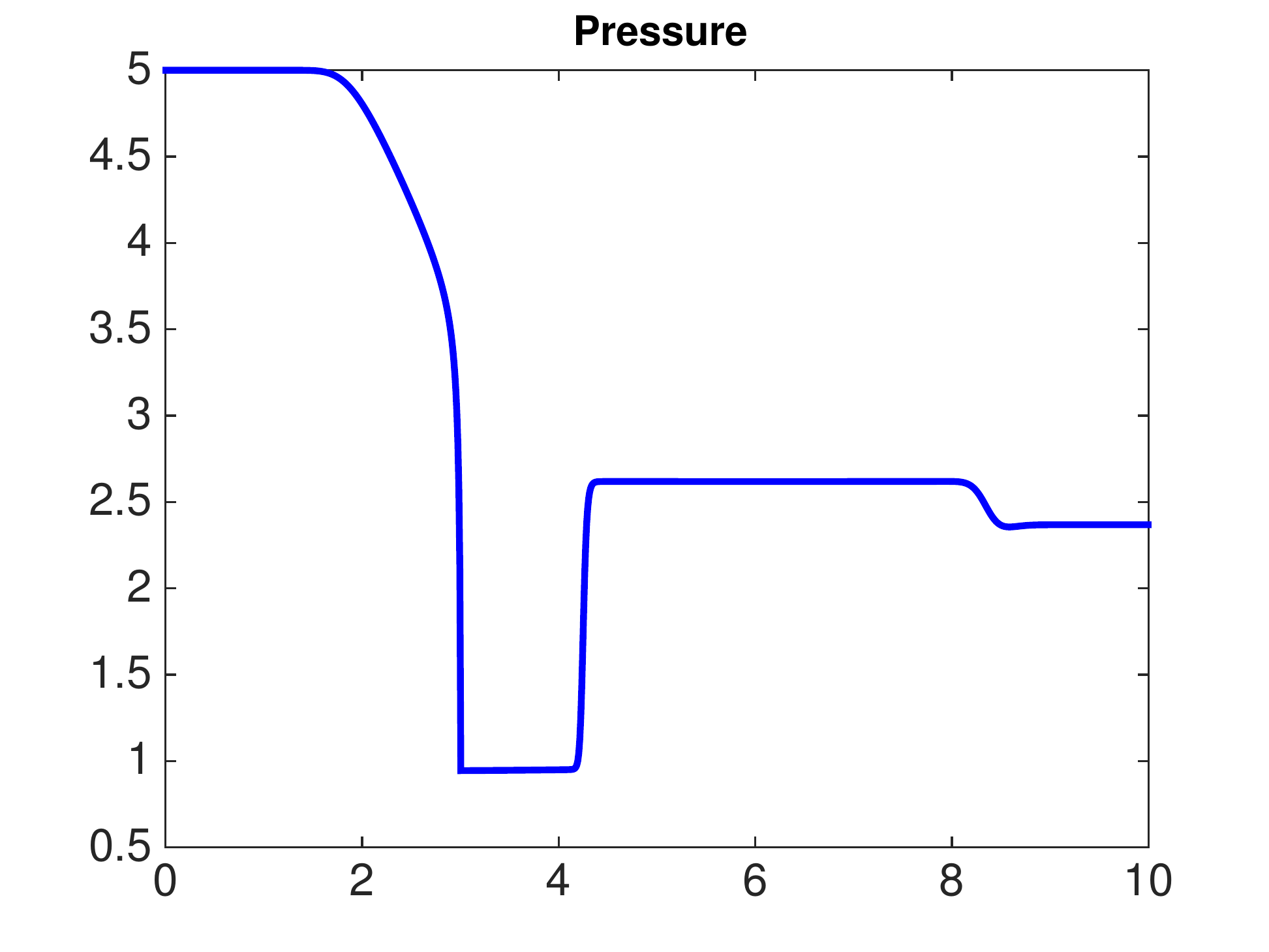}
\end{minipage}
}
\caption*{Fig. 4.3. Test 3.}
\end{figure}

\noindent
{\bf Test 4.} The initial data is given by
\begin{equation}\label{4.8}
\begin{array}{lll}
\displaystyle (\rho_-,u_-,p_-,a_0)=(1.075, 1.5, 5.0, 1.0),\quad 0<x<2.9,\\
 (\rho_m,u_m,p_m,a_0)=(1.0, 1.5, 5.0, 1.0), \quad 2.9<x<3,\\
 (\rho_+,u_+,p_+,a_1)=(1.0687, 0.9357, 4.3777, 1.5), \quad 3<x<10.
\end{array}
\end{equation}

We have $U_-\in D_2$, $U_m\in D_2$, $U_+\in D_2$. The result is shown at $t=1.0s$, see Fig. 4.4. The solution begins with a backward rarefaction wave, followed by a stationary wave, followed by a contact discontinuity, then followed by forward shock wave. The result is the same with that in case 3.

\begin{figure}[htbp]
\subfigure{
\begin{minipage}[t]{0.31\textwidth}
\centering
\includegraphics[width=0.95\textwidth]{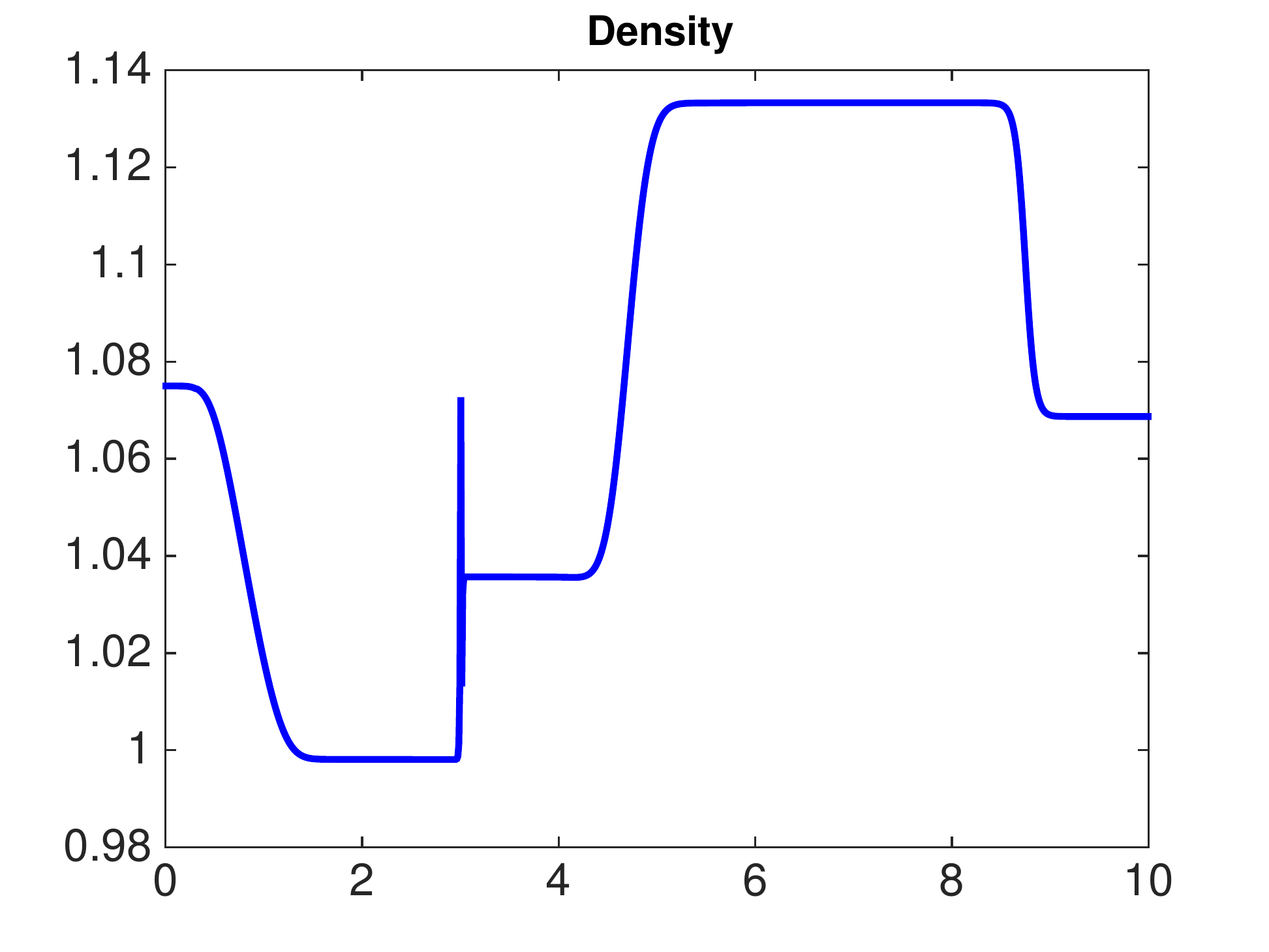}
\end{minipage}
}
\subfigure{
\begin{minipage}[t]{0.31\textwidth}
\centering
\includegraphics[width=0.95\textwidth]{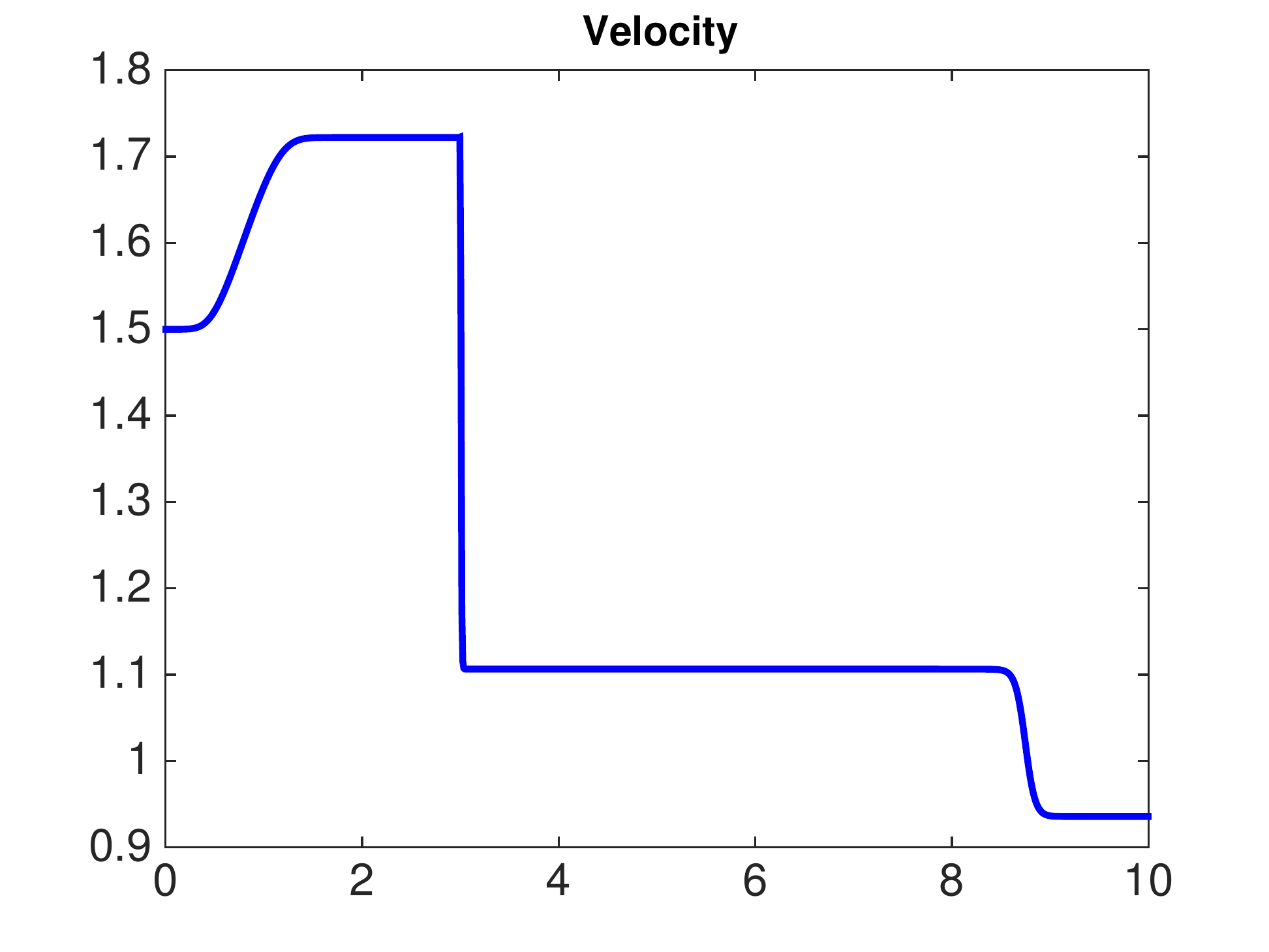}
\end{minipage}
}
\subfigure{
\begin{minipage}[t]{0.31\textwidth}
\centering
\includegraphics[width=0.95\textwidth]{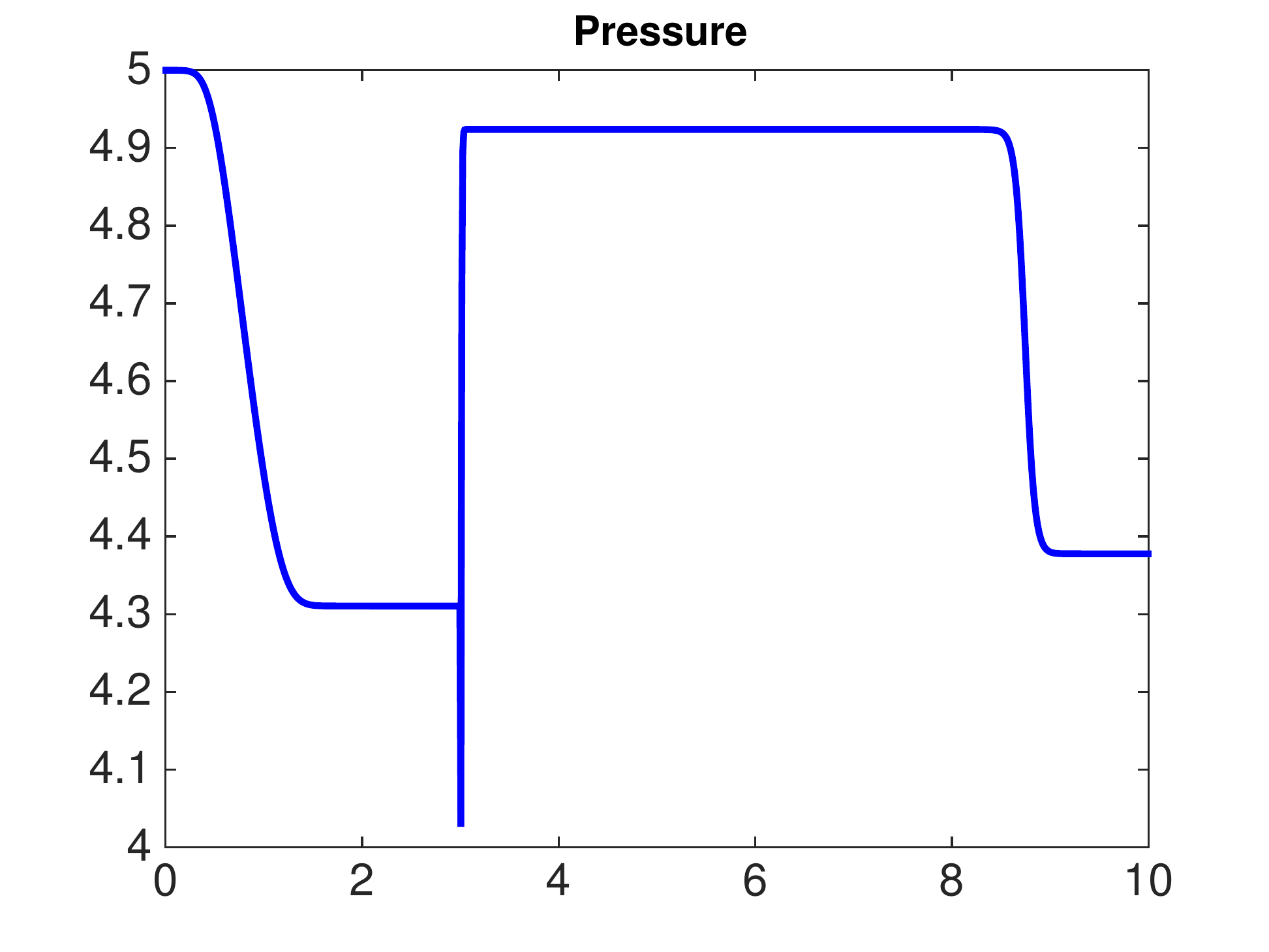}
\end{minipage}
}
\caption*{Fig. 4.4. Test 4.}
\end{figure}

\noindent
{\bf Test 5.} The initial data is given by
\begin{equation}\label{4.9}
\begin{array}{lll}
\displaystyle (\rho_-,u_-,p_-,a_0)=(1.0, 4.0, 10.0, 1.0),\quad 0<x<2.9,\\
 (\rho_m,u_m,p_m,a_0)=(1.2, 4.0, 10.0, 1.0), \quad 2.9<x<3,\\
 (\rho_+,u_+,p_+,a_1)=(1.63872, 1.9527, 18.6486, 1.5), \quad 3<x<10.
\end{array}
\end{equation}

We have $U_-\in D_2$, $U_m\in D_2$, $U_+\in D_2$. The result is shown at $t=0.5s$, see Fig. 4.5. The solution begins with a backward shock wave, followed by a stationary wave, followed by a contact discontinuity, then followed by a forward rarefaction wave. The result is the same with that in case 4.

\begin{figure}[htbp]
\subfigure{
\begin{minipage}[t]{0.31\textwidth}
\centering
\includegraphics[width=0.95\textwidth]{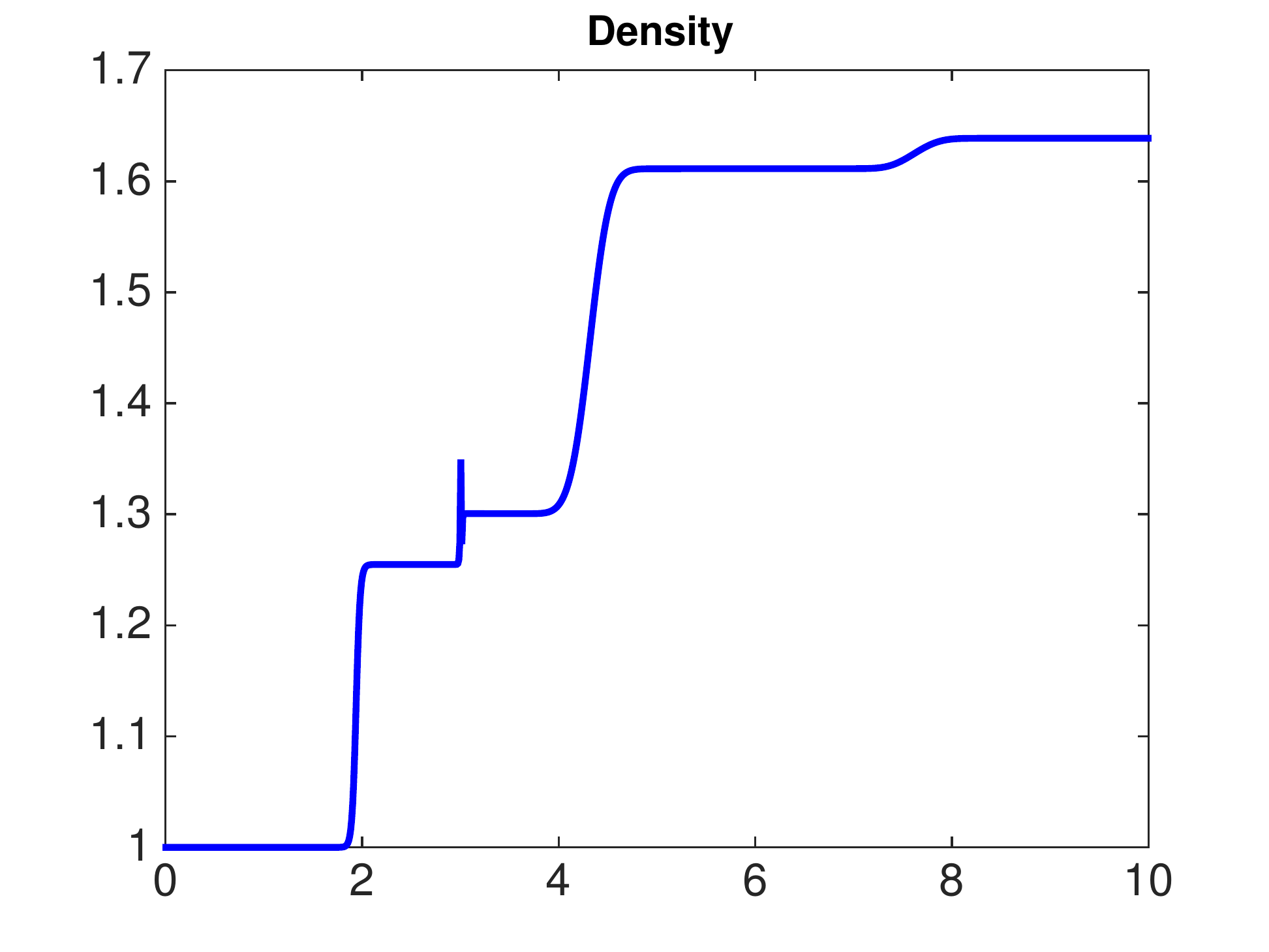}
\end{minipage}
}
\subfigure{
\begin{minipage}[t]{0.31\textwidth}
\centering
\includegraphics[width=0.95\textwidth]{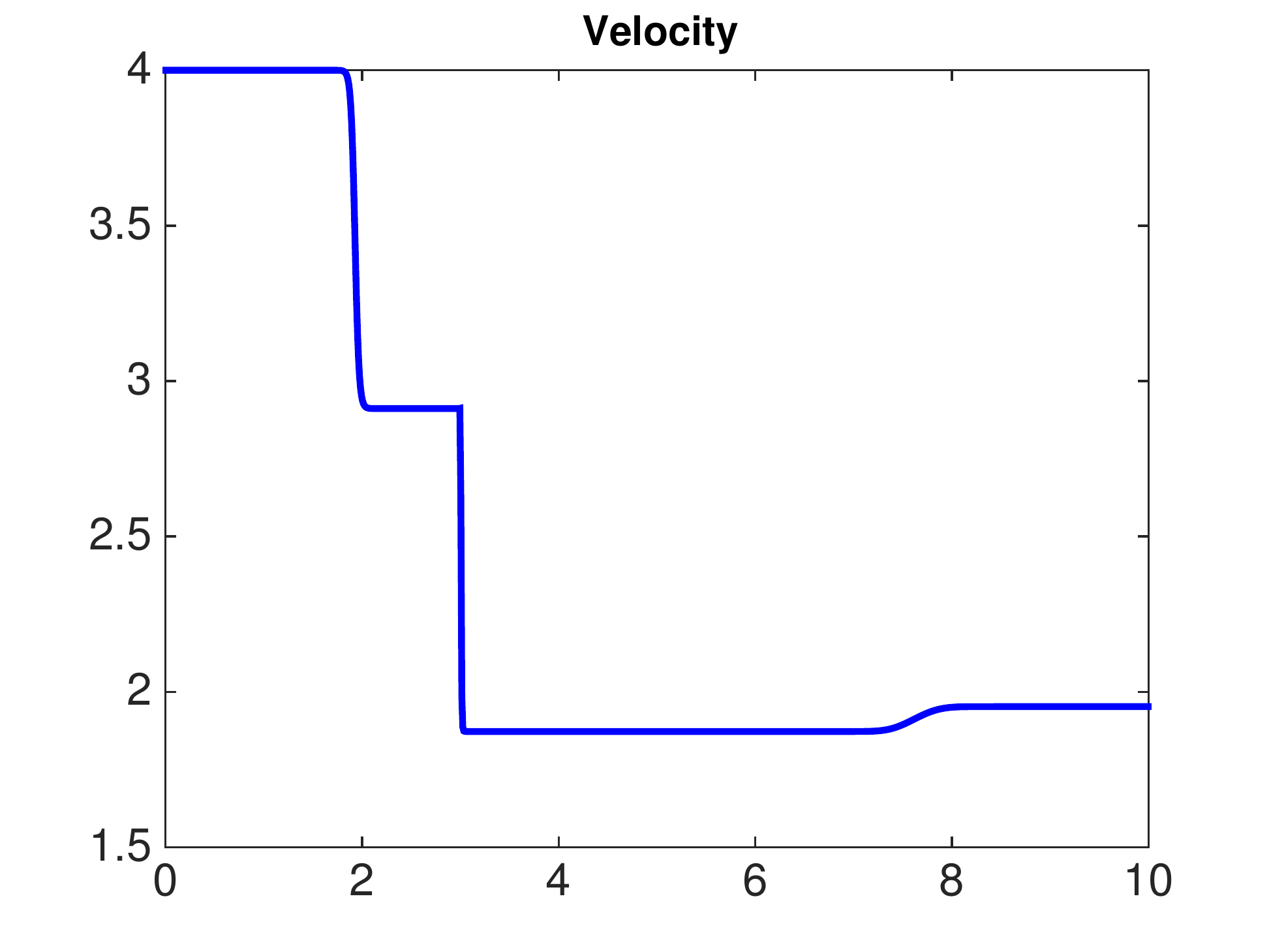}
\end{minipage}
}
\subfigure{
\begin{minipage}[t]{0.31\textwidth}
\centering
\includegraphics[width=0.95\textwidth]{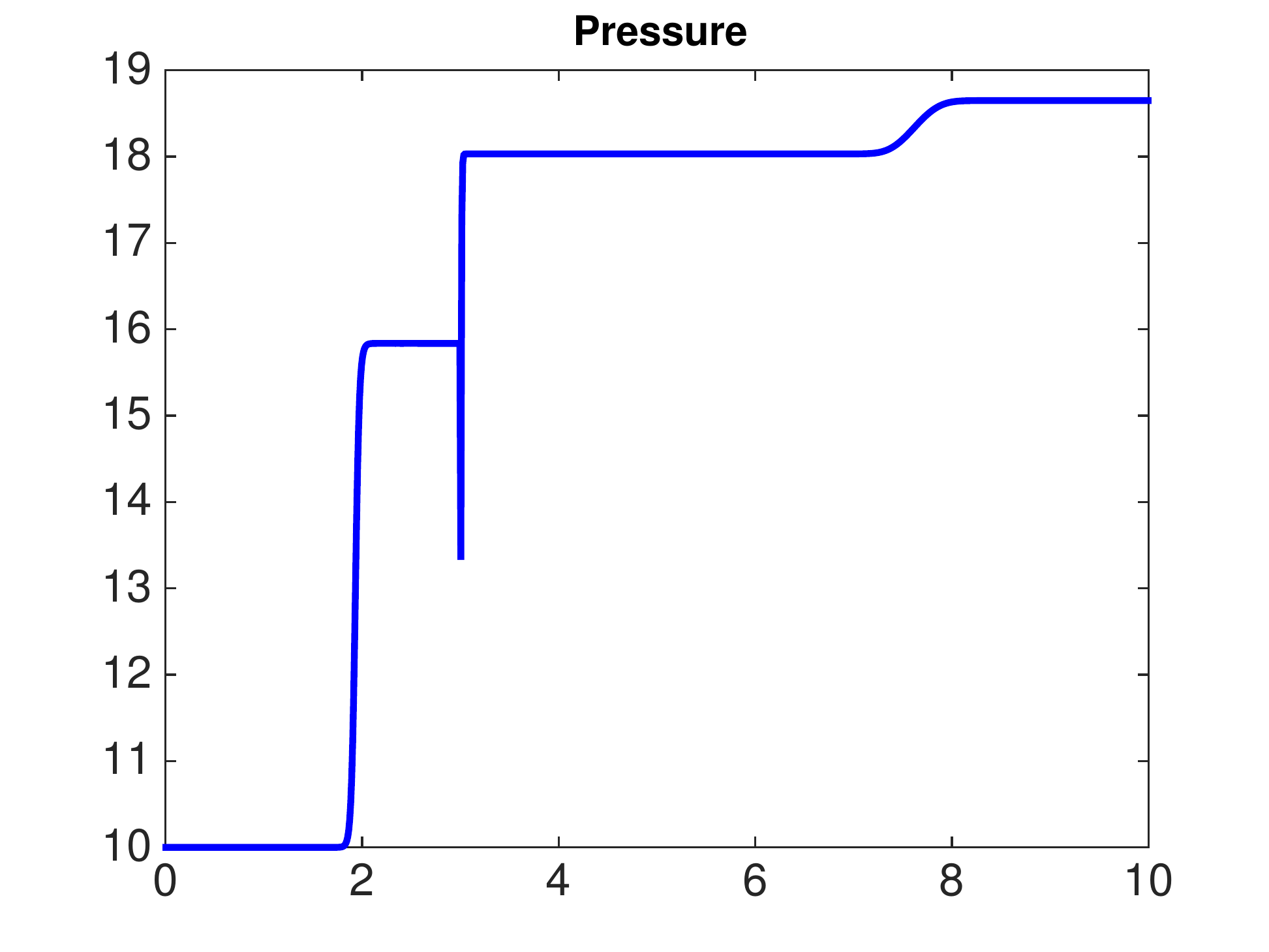}
\end{minipage}
}
\caption*{Fig. 4.5. Test 5.}
\end{figure}

\noindent
{\bf Test 6.} The initial data is given by
\begin{equation}\label{4.10}
\begin{array}{lll}
\displaystyle (\rho_-,u_-,p_-,a_0)=(7.0, 1.5, 5.0, 1.0),\quad 0<x<2.9,\\
 (\rho_m, u_m,p_m,a_0)=(1.0, 1.5, 5.0, 1.0), \quad 2.9<x<3,\\
 (\rho_+, u_+,p_+,a_1)=(1.0687, 0.9375, 4.3777, 1.5), \quad 3<x<10.
\end{array}
\end{equation}

We have $U_-\in D_1$, $U_m\in D_2$, $U_+\in D_2$. The result is shown at $t=1.0s$, see Fig. 4.6. The solution begins with a stationary wave, followed by a backward shock wave, followed by a contact discontinuity, then followed by a forward shock wave. The result is the same with that in the transonic case.

\begin{figure}[htbp]
\subfigure{
\begin{minipage}[t]{0.31\textwidth}
\centering
\includegraphics[width=0.95\textwidth]{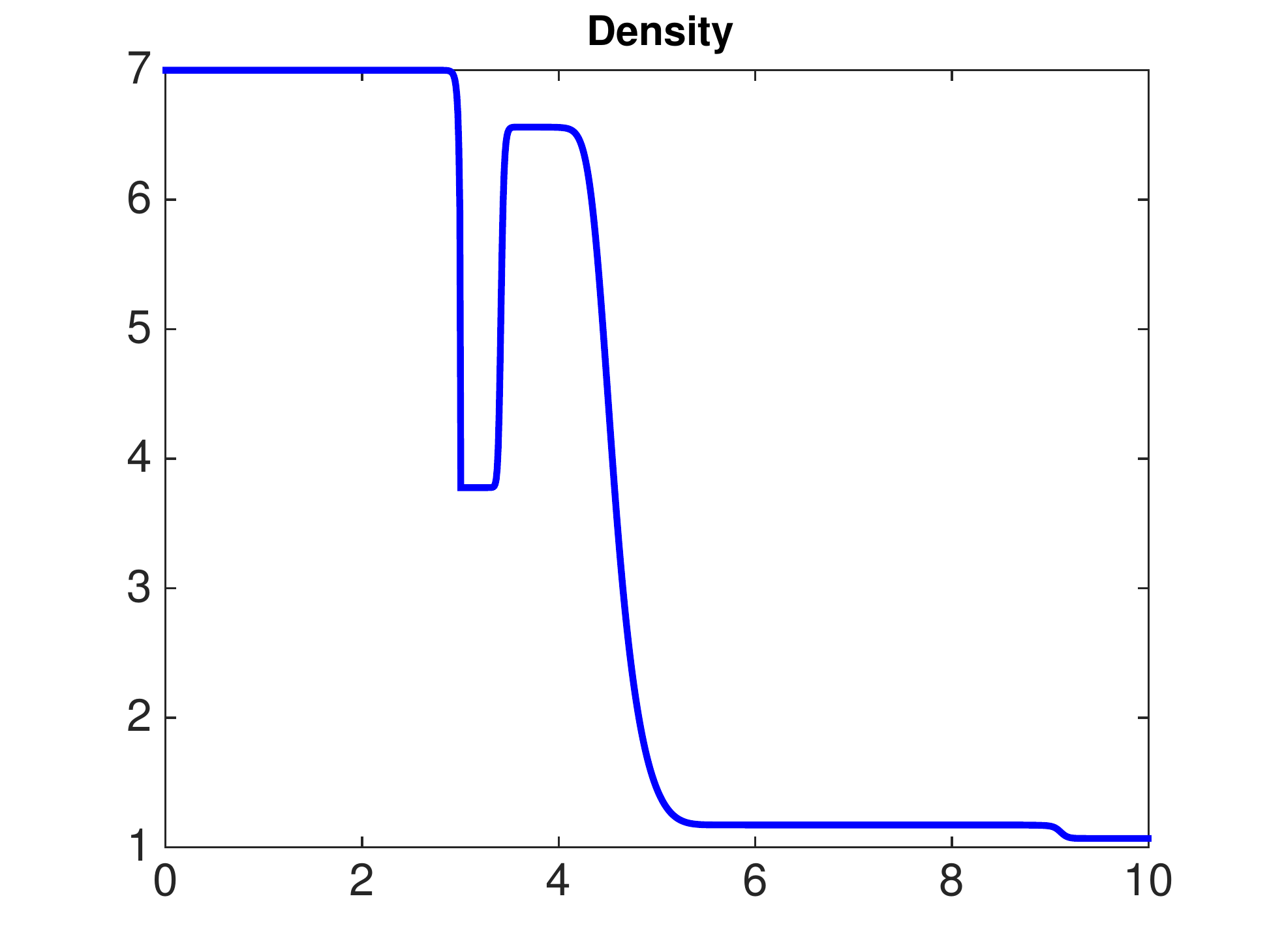}
\end{minipage}
}
\subfigure{
\begin{minipage}[t]{0.31\textwidth}
\centering
\includegraphics[width=0.95\textwidth]{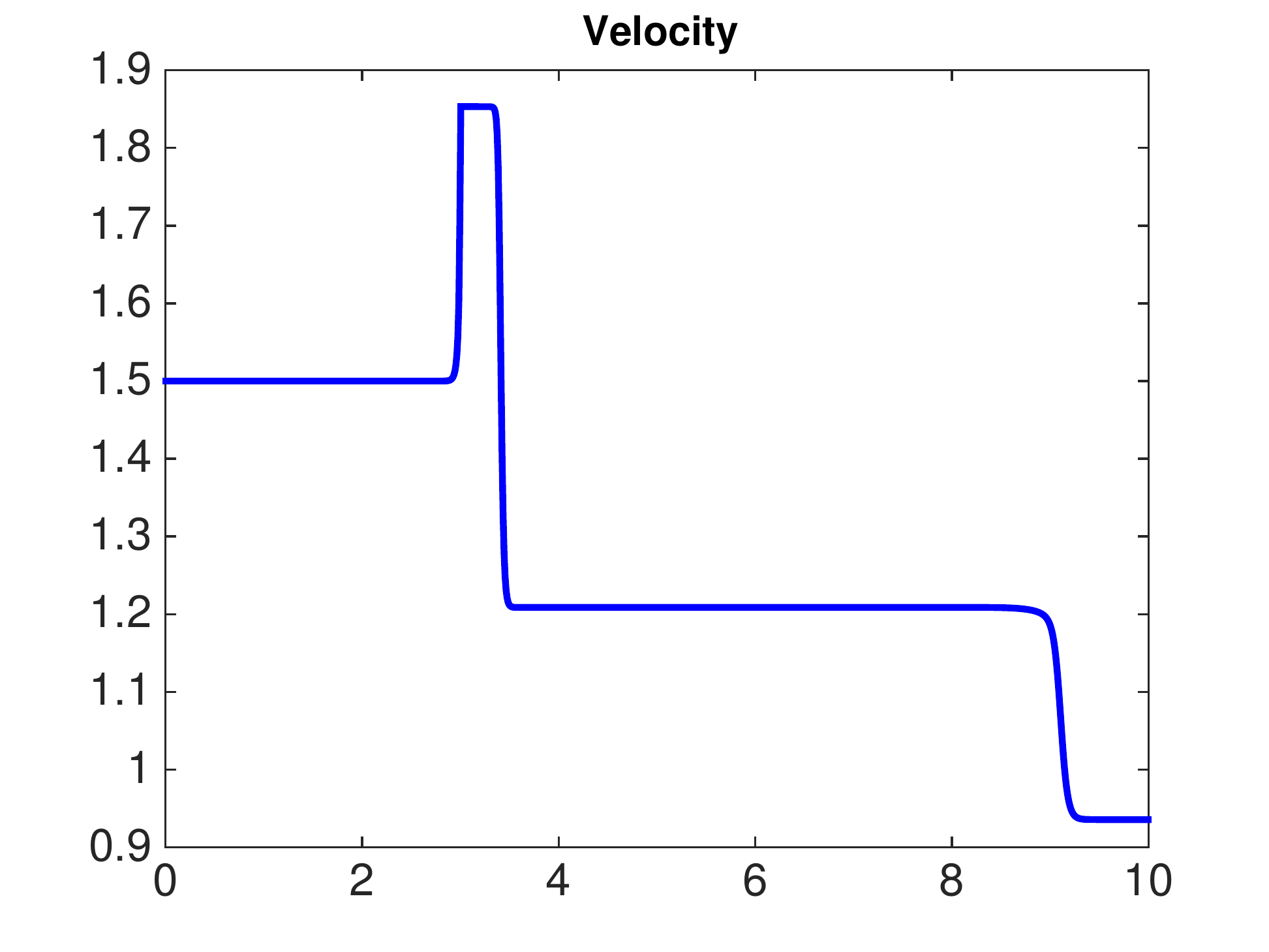}
\end{minipage}
}
\subfigure{
\begin{minipage}[t]{0.31\textwidth}
\centering
\includegraphics[width=0.95\textwidth]{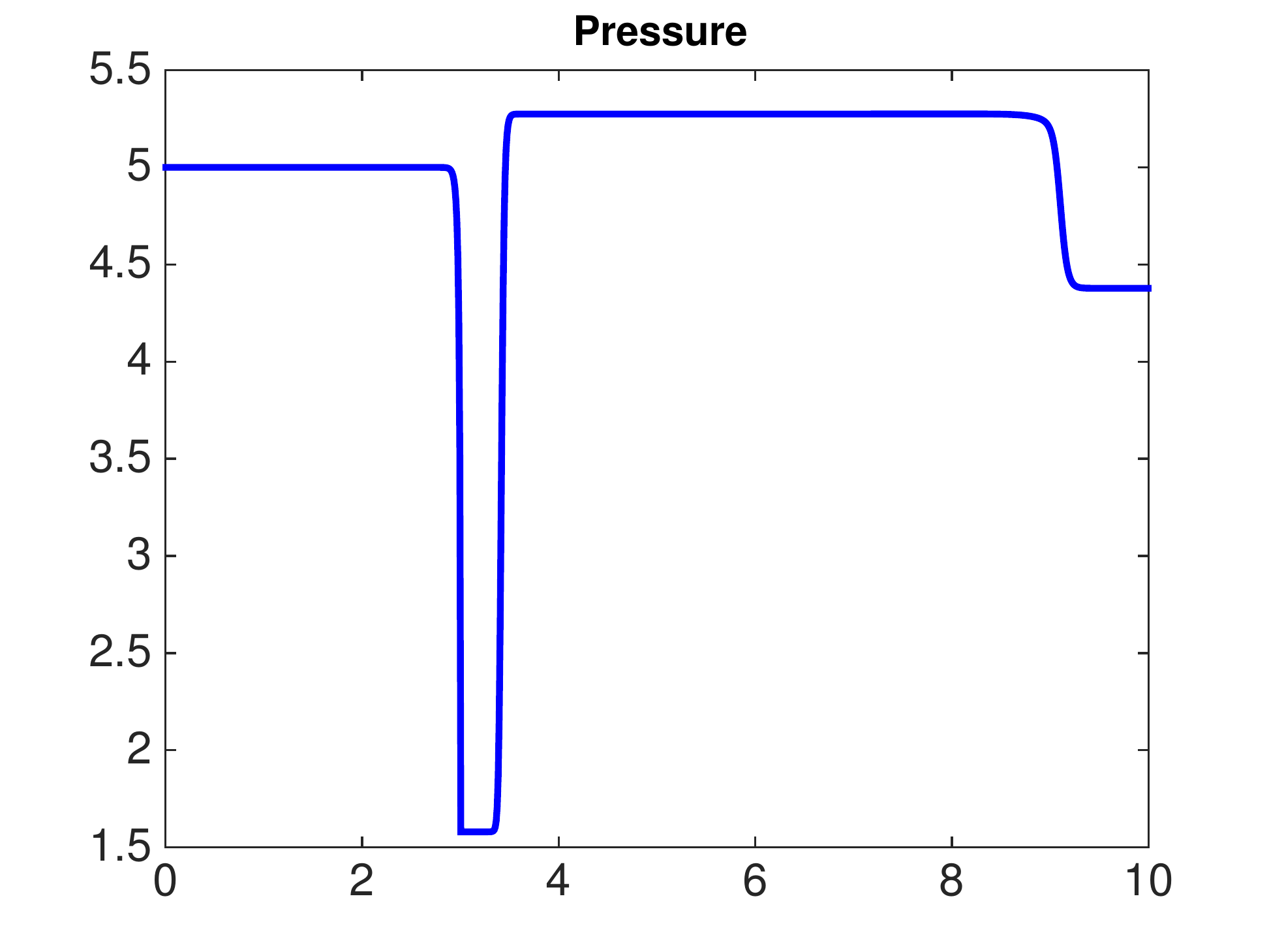}
\end{minipage}
}
\caption*{Fig. 4.6. Test 6.}
\end{figure}

\noindent
{\bf Test 7.} The initial data is given by
\begin{equation}\label{4.11}
\begin{array}{lll}
\displaystyle (\rho_-,u_-,p_-,a_0)=(1.5, 4.0, 10.0, 1.0),\quad 0<x<2.9,\\
 (\rho_m,u_m,p_m,a_0)=(1.2, 4.0, 10.0, 1.0), \quad 2.9<x<3,\\
 (\rho_+,u_+,p_+,a_1)=(1.63872,1.9527, 18.6486, 1.5), \quad 3<x<10.
\end{array}
\end{equation}

We have $U_-\in D_1$, $U_m\in D_2$, $U_+\in D_2$. The result is shown at $t=0.5s$, see Fig. 4.7. The solution begins with a backward shock wave, followed by a stationary  wave, followed by a contact discontinuity, then followed by a forward shock wave. The result is the same with that in the transonic case.

\begin{figure}[htbp]
\subfigure{
\begin{minipage}[t]{0.31\textwidth}
\centering
\includegraphics[width=0.95\textwidth]{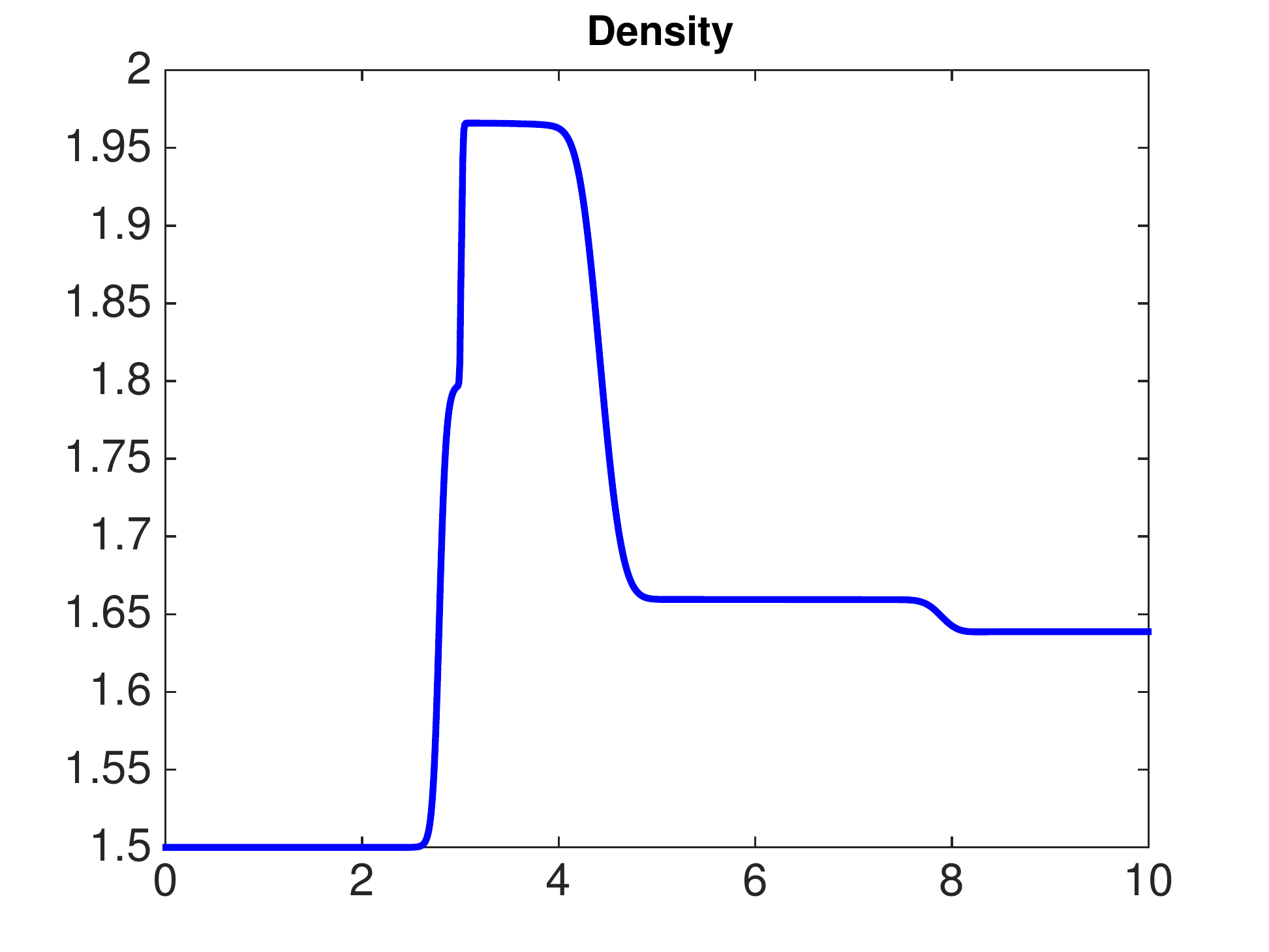}
\end{minipage}
}
\subfigure{
\begin{minipage}[t]{0.31\textwidth}
\centering
\includegraphics[width=0.95\textwidth]{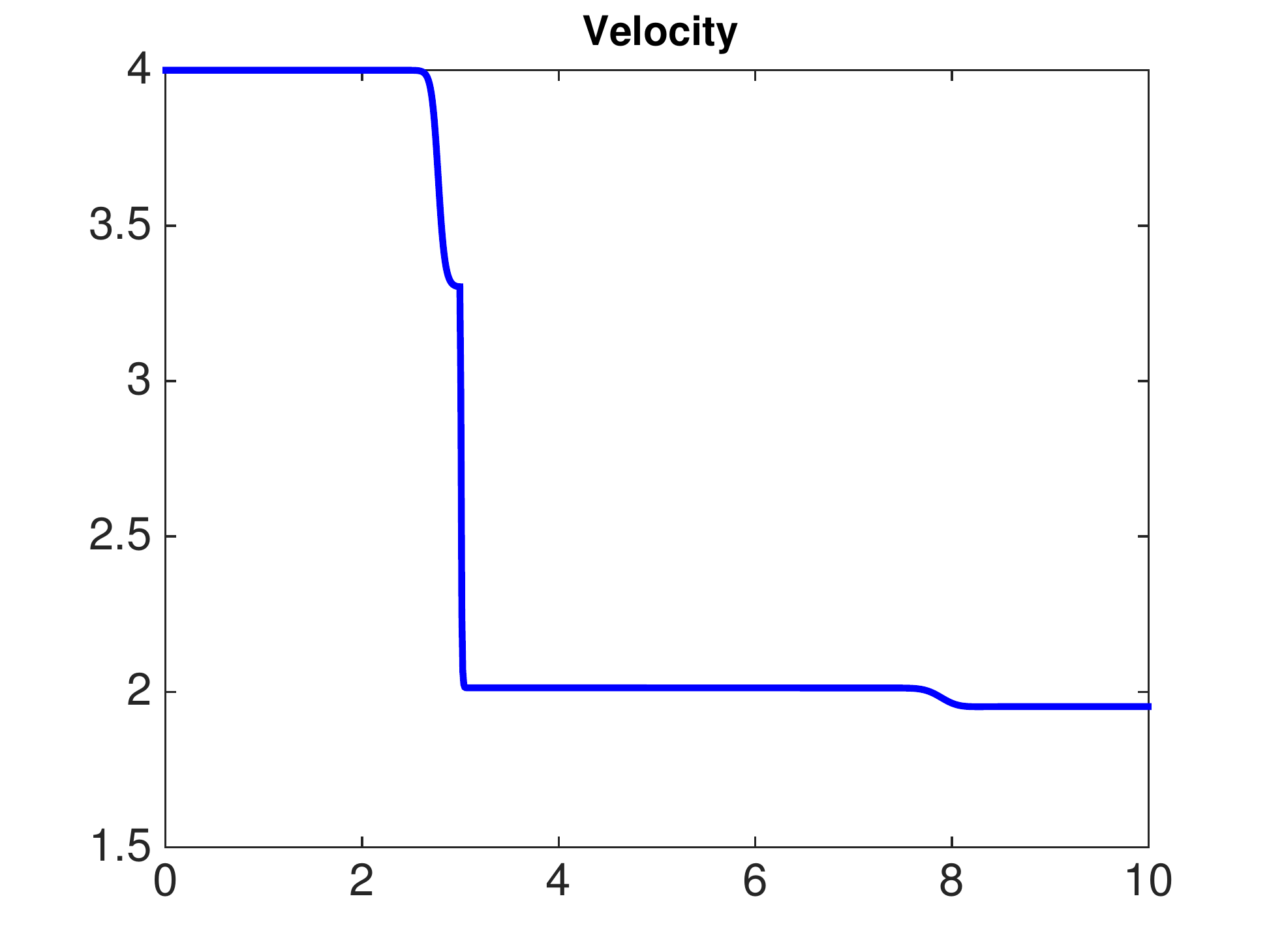}
\end{minipage}
}
\subfigure{
\begin{minipage}[t]{0.31\textwidth}
\centering
\includegraphics[width=0.95\textwidth]{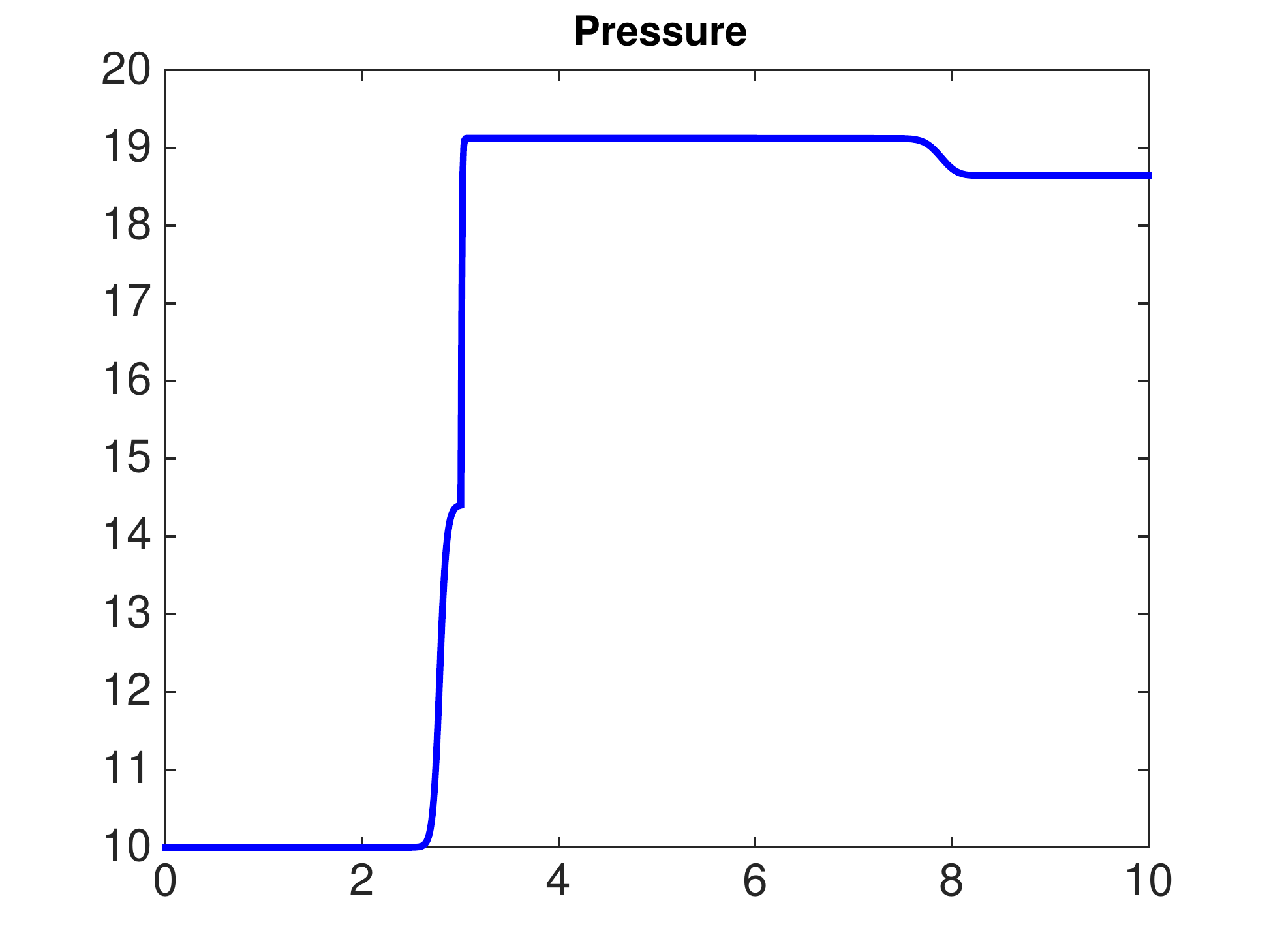}
\end{minipage}
}
\caption*{Fig. 4.7. Test 7.}
\end{figure}

In summary, we mainly obtained the results of contact discontinuity interacts with the stationary wave for nonisentropic flow in a cross-section duct. When the contact discontinuity touches the stationary wave, we need to solve a new Riemann problem with piecewise constant initial data. We classify all the possible cases based on the initial data by using the characteristic analysis method. Numerical results fit well with our analysis in the phase plane.

\vskip2mm
{\bf Acknowledgements} The author wishes to thank Prof. Wancheng Sheng for many valuable discussions about this problem and the method proposed here. The author is also very grateful to the anonymous referees'
for careful reading and comments, which improved the original
manuscript greatly.

\end{document}